\documentclass[letterpaper,     % paper
11pt,                   % font size
dvipsnames,
]{article}%
\usepackage[margin=1in]{geometry}
\usepackage{cancel}
\usepackage[normalem]{ulem}
\PassOptionsToPackage{utf8}{inputenc}
\usepackage{inputenc}
\usepackage{csquotes}
\usepackage{amsmath}
\usepackage[USenglish]{babel}
\usepackage{amsthm, amsfonts, amssymb}
\usepackage{mathtools}
\usepackage{aligned-overset}
\usepackage{doi}
\usepackage{wrapfig}
\usepackage[noend, algosection]{algorithm2e}
\usepackage{paralist}
\usepackage{soul}
\usepackage{aligned-overset}
\usepackage[shortlabels]{enumitem}
\usepackage{thmtools}
\usepackage{comment}
\usepackage[dvipsnames]{xcolor}

\usepackage{tikz}
\usetikzlibrary{positioning, shapes.geometric}
\hypersetup{% setup the hyperref-package options
	pdftitle={An Approximate Version of the Strong Nine Dragon Tree Conjecture},    %   - title (PDF meta)
	pdfsubject={Preprint},%   - subject (PDF meta)
	pdfauthor={Sebastian Mies, Benjamin Moore},    %   - author (PDF meta)
	plainpages=false,           %   -
	colorlinks=false,           %   - colorize links?
	linkcolor=ctcolormain,      %   - link color (e.g., TOC)
	citecolor=ctcolormain,      %   - cite color
	pdfborder={0 0 0},          %   -
	breaklinks=true,            %   - allow line break inside links
	bookmarksnumbered=true,     %
	bookmarksopen=true          %
}

\graphicspath{{images/}}

\definecolor{lightBlue}{RGB}{136, 247, 244}

\def\pathIn#1#2#3{P^{#1}(#2, #3)}
\def\blueTree{\mathcal B_b}
\def\redForest{\mathcal R}

\def\pathInBlueTree#1#2#3{\pathIn{\blueTree(#1)}{#2}{#3}}

\usepackage{todonotes}
\usepackage{caption}
\usepackage[style=numeric-comp, maxbibnames=99]{biblatex}

\tikzset{
	string/.style={},
	redDot/.style={circle, draw, red, fill=red, inner sep=1.5pt},
	blackDot/.style={circle, draw, fill, inner sep=1.5pt},
	transpBetweenUndir/.style={inner sep=0.1pt},
	transp/.style={inner sep=1.5pt},
	blueArc/.style={line width = 0.5mm, cyan, -stealth},
	blueArcRev/.style={line width = 0.5mm, cyan, stealth-},
	redEdge/.style={red, line width=0.5mm},
	redDotted/.style={red, densely dotted, line width=0.5mm},
	blueDotted/.style={cyan, loosely dotted, line width=0.5mm, -stealth},
	blueDottedUndir/.style={cyan, loosely dotted, line width=0.5mm},
}

\definecolor{lightBlue}{RGB}{136, 247, 244}

\def\singleFigure#1#2#3{
    \begin{figure}[!htb]
        \captionsetup{singlelinecheck=off}
        \center
        #1
        \caption[]{#2}
        #3
    \end{figure}
}
\def\doubleFigure#1#2#3#4{
    \begin{figure}[!htb]
        \minipage{0.49\textwidth}
        \raggedright
        #1
        \endminipage\hfill
        \minipage{0.49\textwidth}
        \raggedleft
        #2
        \endminipage
        \caption{#3}
        #4
    \end{figure}
}
\def\doubleFigureCentered#1#2#3#4{
    \begin{figure}[!htb]
        \minipage{0.49\textwidth}
        \centering
        #1
        \endminipage\hfill
        \minipage{0.49\textwidth}
        \centering
        #2
        \endminipage
        \caption{#3}
        #4
    \end{figure}
}
\def\tripleFigure#1#2#3#4#5{
    \begin{figure}[!htb]
        \minipage{0.32\textwidth}
        #1
        \endminipage\hfill
        \minipage{0.32\textwidth}
        #2
        \endminipage\hfill
        \minipage{0.32\textwidth}%
        #3
        \endminipage
        \caption{#4}
        #5
    \end{figure}
}
\def\myTikZ#1{
    \begin{tikzpicture}[scale=\scale, baseline=-1mm, -,]
        #1
    \end{tikzpicture}
}
\def\scale{0.46}
\def\lengthEdge{1.7}

\def\xX{(2+2)}
\def\xTwoY{1}
\def\childDegree{20}
\def\vX{2}
\def\epsForL{0.1}
\def\xFourY{(\xTwoY + 3)}
\def\vStarY{(\xTwoY + 2)}

\newcommand{\nodeWithLabel}[4]{ % label, orientation, distance
    \node[#1, label={[label distance=#4]#3:#2}]
}

\def\child#1#2#3#4#5#6{ % name of gen, offsetX (in lengthEdge) from gen vertex, offsetY, degrees, size of Child, colour of child
    \node[redDot] (#1 c) at ({#2*\lengthEdge + cos(#4)*\lengthEdge}, {#3*\lengthEdge + sin(#4)*\lengthEdge}) {};
    \ifnum#5>0
    \node[redDot] (#1 l) at ({#2*\lengthEdge + (cos(#4) + cos(#4+\childDegree))*\lengthEdge}, {#3*\lengthEdge + (sin(#4) + sin(#4+\childDegree))*\lengthEdge}) {};
    \fi
    \ifnum#5>1
    \node[redDot] (#1 r) at ({#2*\lengthEdge + (cos(#4) + cos(#4-\childDegree))*\lengthEdge}, {#3*\lengthEdge + (sin(#4) + sin(#4-\childDegree))*\lengthEdge}) {};
    \fi
    \draw[redEdge]
    \ifnum#5>0
    (#1 c) edge (#1 l)
    \fi
    \ifnum#5>1
    (#1 c) edge (#1 r)
    \fi
    ;
    \draw[#6]
    (#1) edge (#1 c)
    ;
}
\def\nodeWithChild#1#2#3{ % name of gen, offsetX (in lengthEdge) from gen vertex, offsetY, degrees, size of Child, colour of child
}

\def\nodeUWithChild#1{ % label
	\nodeWithLabel{redDot}{#1}{below}{0} (u) at ({-\lengthEdge}, 0) {};
	\child{u}{-1}{0}{120}{2}{redEdge}
}

\def\nodeVPrime#1{ % label
	\nodeWithLabel{redDot}{#1}{below}{-2} (v') at ({(\vX-1)*\lengthEdge}, 0) {};
}
\def\nodeV#1{ % label
	\nodeWithLabel{redDot}{#1}{above}{0} (v) at ({\vX*\lengthEdge}, 0) {};
}

\def\xTwoWithChild#1{ % arc colour
	\nodeWithLabel{redDot}{$x_2$}{left}{-2} (x2) at (0, -1*\xTwoY*\lengthEdge) {};
	\child{x2}{0}{-1*\xTwoY}{270}{1}{#1}
}

\def\xWithChild#1#2{ % label, arc colour
	\nodeWithLabel{redDot}{#1}{below}{0} (x) at ({\xX*\lengthEdge}, 0) {};
	\child{x}{\xX}{0}{0}{2}{#2}
}

\def\yWithChild#1#2{ % label, arc colour
	\nodeWithLabel{redDot}{#1}{left}{-2} (y) at ({\vX*\lengthEdge}, -1*\lengthEdge) {};
	\child{y}{\vX}{-1}{270}{2}{#2}
}
\def\L{
	\node[draw=black, minimum height={(1 + \epsForL)*\lengthEdge*1cm}, minimum width={(0.5*(\xX-\vX) + \epsForL)*\lengthEdge*1cm}, label={290:$L$}] (L) at ({0.5*(\vX + \xX)*\lengthEdge}, 0) {};
}
\def\LPrime{
	\node[draw=black, minimum height={(\xTwoY + \epsForL)*\lengthEdge*1cm}, minimum width={(1 + \epsForL)*\lengthEdge*1cm}, label={240:$L'$}] (L') at (0, 0) {};
}

\def\vStarWithChild#1{% arc colour
	\nodeWithLabel{redDot}{$v_*$}{left}{-2} (v*) at (0, {\vStarY*\lengthEdge}) {};
	\child{v*}{0}{\vStarY}{0}{1}{#1}
}
\def\xFourWithChild{
	\nodeWithLabel{redDot}{$x_4$}{above}{0} (x4) at (0, {\xFourY*\lengthEdge}) {};
	\child{x4}{0}{\xFourY}{180}{2}{blueArcRev}
}

\def\xFourTransp{
	\node[transp] (x4Transp) at (0, {\xFourY*\lengthEdge}) {};
}

\def\xFiveWithChild{
	\nodeWithLabel{redDot}{$x_5$}{right}{0} (x5) at (1*\lengthEdge, {\xFourY*\lengthEdge}) {};
	\child{x5}{1}{\xFourY}{90}{0}{blueArc}
}
\def\r{
	\node[transp] (r) at ({\xX*\lengthEdge}, {\vStarY*\lengthEdge}) {$r$};
}

\def\KPrime{
	\node[transp, text=red] (K') at ({1*\lengthEdge}, {1*\lengthEdge}) {$K'$};
}
\def\KTilde{
	\node[transp, text=red] (KTilde) at ({-1*\lengthEdge}, {\vStarY*\lengthEdge}) {$\tilde K$};
}

\def\commonThings{
	\node[redDot] (c1) at (0, 0) {};
	\node[redDot] (c2) at (0, \xTwoY*\lengthEdge) {};
	\node[redDot] (c3) at (0, {(\xTwoY + 1)*\lengthEdge}) {};
	\node[transp] (tu) at ({-0.6*\lengthEdge}, {0.8*\lengthEdge}) {};
	\node[transp] (tv*) at ({(\xX - 1)*\lengthEdge}, {\vStarY*\lengthEdge}) {};
	\r
	\draw[redEdge]
	(u) edge (c1)
	(c1) edge (v')
	(v) edge (y)
	(c3) edge (v*)
	;
	\draw[redDotted]
	(c1) edge (x2)
	(v) edge (x)
	(c1) edge (c2)
	;
	\draw[blueDottedUndir]
	(u) edge (tu)
	;
	\draw[blueDotted]
	(tu) edge (c2)
	(c3) edge[bend left=40] (v*)
	;
	\draw[blueArc]
	(c2) edge (c3)
	;
}

\def\defIntValidStates{\myTikZ{
		\nodeUWithChild{$u(\mathcal T_a)$}
		\nodeVPrime{}
		\nodeV{}
		\xTwoWithChild{blueArc}
		\xWithChild{$x_3$}{blueArc}
		\yWithChild{$x_1$}{blueArc}
		\vStarWithChild{blueArc}
		\xFourWithChild
		\xFiveWithChild
		\KPrime
		\KTilde
		\commonThings
		\draw[redEdge]
		(v') edge (v)
		(v*) edge (x4)
		(x4) edge (x5)
		;
		\draw[blueDotted]
		(tv*) edge (r)
		;
		\draw[blueDottedUndir]
		(v* c) edge (tv*)
		;
}}

\def\commonOperations{
    \nodeUWithChild{$u(\mathcal T_a)$}
    \L
    \LPrime
    \xFourTransp
    \commonThings
    \draw[redDotted]
        (v*) edge (x4Transp)
    ;
    \draw[blueDotted]
        (c1) edge (tu)
    ;
}

\def\commonCaseOne{
    \nodeVPrime{$v'$}
    \nodeV{$v$}
    \xWithChild{$x$}{blueArc}
    \commonOperations
}

\def\commonSubcaseOneOne{
    \vStarWithChild{blueArc}
    \commonCaseOne
    \draw[blueDotted]
        (v') edge[bend left=20] (tv*)
        (y c) edge[bend right=20] (tv*)
        (tv*) edge (r)
    ;
    \draw[blueDottedUndir]
        (v* c) edge (tv*)
    ;
}

\def\subcaseOneOneBefore{\myTikZ{
    \xTwoWithChild{blueArc}
    \yWithChild{$y$}{blueArc}
    \commonSubcaseOneOne
    \draw[redEdge]
        (v') edge (v)
    ;
    \draw[blueDotted]
        (v) edge[bend left=20] (x2)
        (x2 c) edge[bend right=20] (y)
    ;
}}

\def\subcaseOneOneAfter{\myTikZ{
    \xTwoWithChild{blueArcRev}
    \yWithChild{$y$}{redEdge}
    \commonSubcaseOneOne
    \draw[blueArc]
        (v) edge (v')
    ;
    \draw[blueDotted]
        (x2) edge[bend right=20] (v)
        (y) edge[bend left=20] (x2 c)
    ;
}}

\def\commonSubcaseOneTwo{	
    \yWithChild{$y$}{blueArc}
    \xTwoWithChild{blueArc}
    \commonCaseOne
    \draw[blueDotted]
        (v* c) edge[bend right=10] (y)
        (v) edge[bend left=40] (y)
        (y c) edge[bend right=20] (r)
    ;
}

\def\subCaseOneTwoBefore{\myTikZ{
    \vStarWithChild{blueArc}
    \commonSubcaseOneTwo
    \draw[redEdge]
        (v') edge (v)
    ;
    \draw[blueDotted]
        (v') edge (v*)
    ;
}}

\def\subCaseOneTwoAfter{\myTikZ{
    \vStarWithChild{redEdge}
    \commonSubcaseOneTwo
    \draw[blueArc]
        (v') edge (v)
    ;
    \draw[blueDotted]
        (v*) edge (v')
    ;
}}

\def\xyzXWithChild#1#2{	% label, colour of generating edge
    \nodeWithLabel{redDot}{#1}{below}{0} (x) at (0, 0) {};
    \child{x}{0}{0}{90}{2}{#2}
}
\def\xyzVOneX{4}
\def\xyzVOne#1{	% label
    \nodeWithLabel{redDot}{#1}{below}{0} (v1) at ({\xyzVOneX*\lengthEdge}, 0) {};
}
\def\xyzYWithChild#1#2{	% label, colour of generating edge
    \def\xyzYX{(\xyzVOneX + 1)}
    \nodeWithLabel{redDot}{#1}{right}{0} (y) at ({\xyzYX*\lengthEdge}, 0) {};
    \child{y}{\xyzYX}{0}{270}{1}{#2}
}

\def\xyzZX{1.5}
\def\xyzVTwoY{2}
\def\xyzVTwo#1{	% label
    \nodeWithLabel{redDot}{#1}{left}{0} (v2) at ({\xyzZX*\lengthEdge}, {-1*\xyzVTwoY*\lengthEdge}) {};
}
\def\xyzZWithChild#1{	% label
    \def\xyzZY{(\xyzVTwoY + 1)}
    \nodeWithLabel{redDot}{#1}{below}{0} (z) at ({\xyzZX*\lengthEdge}, {-1*\xyzZY*\lengthEdge}) {};
    \child{z}{\xyzZX}{-1*\xyzZY}{180}{0}{blueArc}
}
\def\xyzCommon{
    \node[transp] (t) at ({\xyzZX*\lengthEdge}, 0) {};
    \node[transp] (s) at ({(\xyzZX + 0.5)*\lengthEdge}, 0.5*\lengthEdge) {};
    \draw[blueDotted]
        (x c) edge[bend left=20] (y)
        (y c) edge[bend left=20] (z)
    ;
    \draw[redDotted]
        (x) edge (v1)
        (v1) edge (y)
        (t) edge (v2)
    ;
}

\def\xyzPicOne{\myTikZ{
    \xyzXWithChild{$v_* = u(\mathcal T^*_{v_*})$}{blueArc}
    \xyzVOne{$\bar v$}
    \xyzYWithChild{$x$}{blueArc}
    \xyzVTwo{}
    \xyzZWithChild{}
    \xyzCommon
    \draw[redEdge]
        (v1) edge (y)
        (v2) edge (z)
    ;
    \draw[blueDotted]
        (v1) edge[bend right=10] (s)
        (s) edge[bend right=10] (x)
        (v2) edge[bend right=20] (s)
    ;
}}

\def\xyzPicTwo{\myTikZ{
    \xyzXWithChild{$u(\mathcal T_{v_*})$}{redEdge}
    \xyzVOne{}
    \xyzYWithChild{$v_*$}{blueArc}
    \xyzVTwo{$\bar v$}
    \xyzZWithChild{$x$}
    \xyzCommon
    \draw[redEdge]
        (v2) edge (z)
    ;
    \draw[blueArc]
        (v1) edge (y)
    ;
    \draw[blueDotted]
        (s) edge[bend left=10] (v1)
        (x) edge[bend left=10] (s)
        (v2) edge[bend right=20] (s)
    ;
}}

\def\xyzPicThree{\myTikZ{
    \xyzXWithChild{$u(\mathcal T'_{v_*})$}{redEdge}
    \xyzVOne{}
    \xyzYWithChild{}{redEdge}
    \xyzVTwo{}
    \xyzZWithChild{$v_*$}
    \xyzCommon
    \draw[blueArc]
        (y) edge (v1)
        (v2) edge (z)
    ;
    \draw[blueDotted]
        (v1) edge[bend right=10] (s)
        (x) edge[bend left=10] (s)
        (s) edge[bend left=20] (v2)
    ;
}}

\def\specialPathsCommon{
    \def\dotOffset{0.35*\lengthEdge}
    \nodeWithLabel{blackDot}{$v_{-1}$}{below}{0} (v-1) at 
    (0, 0) {};
    \nodeWithLabel{blackDot}{$v_0$}{below}{0} (v0) at 
    (\lengthEdge, 0) {};
    \nodeWithLabel{transp}{\textcolor{lightBlue}{$T_{b_1}$}}{above}{2} (threeDots1) at 
    (2*\lengthEdge + \dotOffset, 0) {...};
    \node[blackDot] (a1') at (3*\lengthEdge + 2*\dotOffset, 0) {};
    \node[blackDot] (a1) at (4*\lengthEdge + 2*\dotOffset, 0) {};
    \nodeWithLabel{transp}{\textcolor{cyan}{$T_{b_2}$}}{above}{2} (threeDots2) at 
    (5*\lengthEdge + 3*\dotOffset, 0) {...};
    \node[blackDot] (a2') at (6*\lengthEdge + 4*\dotOffset, 0) {};
    \node[blackDot] (a2) at (7*\lengthEdge + 4*\dotOffset, 0) {};
    \nodeWithLabel{transp}{\textcolor{blue}{$T_{b_3}$}}{above}{2} (threeDots3) at (8*\lengthEdge + 5*\dotOffset, 0) {...};
    \nodeWithLabel{blackDot}{$v_{l-1}$}{below}{0} (a3') at (9*\lengthEdge + 6*\dotOffset, 0) {};
    \nodeWithLabel{blackDot}{$v_l$}{below}{0} (a3) at (10*\lengthEdge + 6*\dotOffset, 0) {};
    \draw[line width = 0.5mm, lightBlue, dotted]
    (a1) edge[bend right=50] (v-1);
    \draw[line width = 0.5mm, cyan, dotted]
    (a2) edge[bend right=50] (a1');
    \draw[line width = 0.5mm, blue, dotted]
    (a3) edge[bend right=50] (a2');
}
\def\specialPaths{\myTikZ{
    \specialPathsCommon
    \draw[redEdge]
        (v-1) edge (v0)
    ;
    \draw[line width = 0.5mm, lightBlue, ->, -stealth]
        (v0) edge (threeDots1)
        (threeDots1) edge (a1')
        (a1') edge (a1)
    ;
    \draw[blueArc]
        (a1) edge (threeDots2)
        (threeDots2) edge (a2')
        (a2') edge (a2)
    ;
    \draw[line width = 0.5mm, blue, ->, -stealth]
        (a2) edge (threeDots3)
        (threeDots3) edge (a3')
        (a3') edge (a3)
    ;
}}

\def\specialPathsAugm{\myTikZ{
    \specialPathsCommon
    \draw[redEdge]
        (a3') edge (a3)
    ;
    \draw[line width = 0.5mm, lightBlue, ->, -stealth]
        (v0) edge (v-1)
        (threeDots1) edge (v0) 
        (a1') edge(threeDots1)
    ;
    \draw[blueArc]
        (a1) edge (a1')
        (threeDots2) edge(a1)
        (a2') edge (threeDots2)
    ;
    \draw[line width = 0.5mm, blue, ->, -stealth]
        (a2) edge (a2')
        (threeDots3) edge (a2)
        (a3') edge (threeDots3)
    ;
}}

\def\angleSwapDef{-60}
\def\xOnABranch#1{#1*cos(\angleSwapDef)*\lengthEdge}
\def\yOnABranch#1{#1*sin(\angleSwapDef)*\lengthEdge}
\def\posOnABranch#1{({\xOnABranch{#1}}, {\yOnABranch{#1}})}

\def\xOnBBranch#1{\xOnABranch{0} + #1*cos(180-\angleSwapDef)*\lengthEdge}
\def\yOnBBranch#1{\yOnABranch{0} + #1*sin(180-\angleSwapDef)*\lengthEdge}
\def\posOnBBranch#1{({\xOnBBranch{#1}}, {\yOnBBranch{#1}})}

\def\xOnCBranch#1{\xOnBBranch{2} + #1*cos(\angleSwapDef)*\lengthEdge}
\def\yOnCBranch#1{\yOnBBranch{2} + #1*sin(\angleSwapDef)*\lengthEdge}
\def\posOnCBranch#1{({\xOnCBranch{#1}}, {\yOnCBranch{#1}})}

\def\swapDefBasic{
    \nodeWithLabel{blackDot}{$r$}{above}{0} (r) at (0, 0) {};
    \node[blackDot] (a1) at \posOnABranch{1} {};
    \nodeWithLabel{blackDot}{$v'_1$}{left}{-2} (v'1) at \posOnABranch{2} {};
    \nodeWithLabel{blackDot}{$u'_1$}{left}{-2} (u'1) at \posOnABranch{3} {};
    \nodeWithLabel{blackDot}{$u_1$}{left}{-2} (u1) at \posOnABranch{4} {};
    \node[blackDot] (a2) at \posOnABranch{5} {};
    \nodeWithLabel{blackDot}{$v_1$}{left}{-2} (v1) at \posOnABranch{6} {};
    \node[blackDot] (b1) at \posOnBBranch{1} {};
    \nodeWithLabel{blackDot}{$u'_2$}{left}{-1} (u'2) at \posOnBBranch{2} {};
    \nodeWithLabel{blackDot}{$u_2$}{left}{-1} (u2) at \posOnBBranch{3} {};
    \nodeWithLabel{blackDot}{$v_2$}{left}{-1} (v2) at \posOnBBranch{4} {};
    \node[blackDot] (b2) at \posOnBBranch{5} {};
    \nodeWithLabel{blackDot}{$v'_2$}{right}{-1} (v'2) at \posOnCBranch{1} {};
    \node[blackDot] (c1) at \posOnCBranch{2} {};
    \draw[blueArc]
        (a1) edge (r)
        (v'1) edge (a1)
        (u'1) edge (v'1)
        (b1) edge (r)
        (u'2) edge (b1)
        (b2) edge (v2)
        (v'2) edge (u'2)
        (c1) edge (v'2)
    ;
}

\def\swapDefBefore{\myTikZ{
    \swapDefBasic
    \draw[blueArc]
        (u1) edge (u'1)
        (u2) edge (u'2)
        (a2) edge (u1)
        (v1) edge (a2)
        (v2) edge (u2)
    ;
    \draw[redEdge]
        (v1) edge[bend right=40] (v'1)
        (v2) edge (v'2)
    ;
}}
\def\swapDefAfter{\myTikZ{
    \swapDefBasic
    \draw[blueArc]
        (v1) edge[bend right=40] (v'1)
        (v2) edge (v'2)
        (u1) edge (a2)
        (a2) edge (v1)
        (u2) edge (v2)
    ;
    \draw[redEdge]
        (u1) edge (u'1)
        (u2) edge (u'2)
    ;
}}

\author{Sebastian Mies\thanks{Institute of Computer Science, Johannes Gutenberg University Mainz, email: smies@students.uni-mainz.de}   \, and Benjamin Moore\thanks{Institute of Science and Technology Austria, email: Benjamin.Moore@ist.ac.at. Benjamin Moore is supported by ERC Starting Grant “RANDSTRUCT” No. 101076777 and appreciates the gracious support.}}

\title{An Approximate Version of the Strong Nine Dragon Tree Conjecture}

\bibliography{bib.bib}

\DeclarePairedDelimiter\bigFloor{\big\lfloor}{\big\rfloor}

\DeclarePairedDelimiter\bigCeil{\big\lceil}{\big\rceil}

\DeclarePairedDelimiter\biggCeil{\bigg\lceil}{\bigg\rceil}

\DeclarePairedDelimiter\bigParan{\big(}{\big)}
\DeclarePairedDelimiter\BigParan{\Big(}{\Big)}

\newcommand\fracArb{\gamma}
\newcommand\explSG{H_{\mathcal T^*}}
\newcommand\density{\frac{d}{d+k+1}}

\def\inOneToK{\in \{1, \ldots, k\}}
\def\inZeroToEll{\in \{0, \ldots, \ell\}}

\declaretheorem{theorem} 
\declaretheoremstyle[
spaceabove=-2em,
spacebelow=6pt,
headfont=\normalfont\itshape,
postheadspace=1em,
qed=\qedsymbol
]{proofStyle}
\declaretheoremstyle[
spaceabove=1em,
spacebelow=6pt,
headfont=\normalfont\itshape,
postheadspace=1em,
qed=\qedsymbol
]{proofOfMainTheoremStyle}
\declaretheoremstyle[
spaceabove=-2em,
spacebelow=6pt,
headfont=\normalfont\itshape,
postheadspace=1em,
qed=\hfill \textit{\color{gray}(End of proof of the claim) }$\;\blacksquare$
]{proofInProofStyle}

\declaretheorem[name={Proof},style=proofInProofStyle,unnumbered,
]{proofInProof}

\declaretheorem[name={\textbf{Proof of Theorem \ref{thm:approxSNDTC}}},style=proofOfMainTheoremStyle,unnumbered,
]{proofOfMainTheorem}
\newtheorem{thm}{Theorem}[section] % reset theorem numbering for each chapter
\newtheorem{lemma}[thm]{Lemma}
\newtheorem{conj}[thm]{Conjecture}
\newtheorem{definition}[thm]{Definition}
\newtheorem{obs}[thm]{Observation}
\newtheorem{corollary}[thm]{Corollary}
\newtheorem*{claimUnnumbered}{Claim}

\newtheorem{notation}[thm]{Notation}

\date{}

\def\maxOfDAndEOfK{\max \{d', e(K)\}}
\def\edgeLimit{\maxOfDAndEOfK - \ell' - 1}
\def\X{X_{\ell'}(K)}
\def\generatingEdgesInRed{\bigcup_{x \in \X} \{xx'\}}
\def\mathcalC{\mathcal C_{\ell'}(K)}
\def\setOfValidDecomps{\mathcal F_{K, \ell'}}
\def\I{\mathcal I_{\ell'}(K)}
\def\iSigStar{i_{\sigma^*}}
\def\ISmallerISigK{\I < \iSigStar(K)}
\def\IEqISigK{\I = \iSigStar(K)}

\def\valueOfEllInText{\bigFloor{\frac{d-1}{k+1}}}
\def\valueOfEllPlusOneInText{\bigCeil{\frac{d}{k+1}}}
\def\dPrimeWithoutUsingEll{d + \bigCeil{k \valueOfEllInText \big(\frac{d}{k+1} - \frac{1}{2}\valueOfEllPlusOneInText \big)}}
\def\dPrimeUsingEll{d + \bigCeil{k \ell \big(\frac{d}{k+1} - \frac{1}{2}(\ell + 1) \big)}}

\begin{document}
\maketitle

\begin{abstract}
    We prove the Strong Nine Dragon Tree Conjecture is true if we replace the edge bound with  $\dPrimeWithoutUsingEll \leq d + \frac{k}{2} \cdot \big(\frac{d}{k+1}\big)^2$.
    More precisely:
    let $G$ be a graph, let $d$ and $k$ be positive integers and $\fracArb(G) = \max_{H \subseteq G, v(H) \geq 2} \frac{e(H)}{v(H) - 1}$. If $\gamma(G) \leq k + \frac{d}{d + k + 1}$, then there is a partition of $E(G)$ into $k + 1$ forests, where in one forest every connected component has at most $\dPrimeWithoutUsingEll$ edges.
\end{abstract}
	
\section{Introduction}
In this paper all graphs may contain parallel edges but are not allowed to contain loops.
This is a paper on graph decompositions. Throughout, $E(G)$ and $V(G)$ denote the edge set and vertex set of a graph $G$, and we use the notation $e(G) = |E(G)|$ and $v(G) = |V(G)|$. A \textit{decomposition} of a graph $G$ is a partition of $E(G)$ into subgraphs. Naturally, we are interested in cases where we can partition $E(G)$ into few subgraphs, each of which are simple. An obvious candidate for a ``simple" graph is a forest. Nash-Williams characterized exactly when a graph decomposes into $k$ forests. To state his theorem, we need a definition.

\begin{definition}
    Given a graph $G$, we let the \textit{fractional arboricity} of a graph, denoted $\fracArb(G)$, be:
    $\fracArb(G) = \max_{H \subseteq G, v(H) \geq 2} \frac{e(H)}{v(H) - 1}$.
\end{definition}

\begin{theorem}[Nash-Williams' Theorem \cite{nash}]
    A graph $G$ decomposes into $k$ forests if and only if $\fracArb(G) \leq k$. 
\end{theorem}

This is an extremely pretty theorem, however it fails to capture some of the information obtained in the parameter $\fracArb$. In particular, $\fracArb$ may be non-integral, and if for example, $k \geq  1$, and $\fracArb(G) = k + \varepsilon$ for $\varepsilon >0$ but small, one might imagine strengthening the decomposition, because intuitively, one only barely needs $k+1$ forests, essentially $k$ forests suffice. There are many possible ways one might try - but for this paper we will try to gain structure on one single forest. We are not the first to do this; at this point in time, there is a large body of literature on theorems of this type. The most relevant for this paper is the Nine Dragon Tree Theorem, which was proven by Jiang and Yang \cite{ndtt}, after a large amount of effort by other authors (see for example, \cite{Kostochkaetal,ndtk2,Yangmatching, sndtck1d2}).

\begin{theorem}[Nine Dragon Tree Theorem \cite{ndtt}]
    Let $G$ be a graph, and let $d,k$ be positive integers. Every graph with $\fracArb(G) \leq k + \density$ decomposes into $k+1$ forests such that one of them has maximum degree $d$.
\end{theorem}

It was shown in \cite{sndtck1d2} that the Nine Dragon Tree Theorem is sharp, in the following sense:

\begin{theorem}[\cite{sndtck1d2}]
    For any positive integers $k$ and $d$ there are arbitrarily large graphs $G$ and a set $S \subseteq E(G)$ of $d+1$ edges such that $\fracArb(G-S) = k + \density$ and $G$ does not decompose into $k+1$ forests where one of the forests has maximum degree $d$. 
\end{theorem}

Despite this, the authors of \cite{sndtck1d2} conjectured a massive strengthening:

\begin{conj}[Strong Nine Dragon Tree Conjecture, \cite{sndtck1d2}]
    Let $G$ be a graph and let $d$ and $k$ be positive integers. If $\fracArb(G) \leq k + \frac{d}{d + k + 1}$, then there is a partition into $k + 1$ forests, where in one forest every connected component has at most $d$ edges.
\end{conj}

The conjecture is known to be true when $d \leq k+1$ \cite{sndtcDLeqKPlusOne}, and  recently it was shown to be true when $d \leq 2(k+1)$ \cite{mies2024strong}. All other cases are open. Before this paper, the best evidence towards the conjecture was that an analogous conjecture for pseudoforests (i.e.\ graphs which have at most one cycle in each connected component) was shown to be true (and in fact much more) in a series of work \cite{ndttPsfs,mies2023pseudoforest,sndtcPsfs}. There is also a digraph version \cite{digraphndt} and a matroidal version \cite{matroidndt} of the Nine Dragon Tree Conjecture. In an effort to keep the introduction short, we refer the reader to \cite{mies2024strong} for a very thorough exposition of the history of the Strong Nine Dragon Tree Conjecture and its variants, as well as applications of the theorem.

Let $f(k,d)$ be a function such that every graph with fractional arboricity at most $k + \density$ decomposes into $k+1$ forests such that one of the forests has every connected component containing at most $f(k,d)$ edges. The Strong Nine Dragon Tree Conjecture says that we can take $f(k,d) = d$. Until now, it was not known if $f(k,d)$ exists for all $k,d$. We show that $f(k,d)$ exists. 

\begin{theorem} \label{thm:approxSNDTC}
    Let $G$ be a graph and let $d$ and $k$ be positive integers. If $\fracArb(G) \leq k + \frac{d}{d + k + 1}$, then there is a partition into $k + 1$ forests, where in one forest every connected component has at most $\dPrimeWithoutUsingEll$ edges.
\end{theorem}

As a remark, we note that almost two years ago, Daqing Yang announced a similar result with Yaqin Zhang and Chenbo Zhu in a conference (ICCM 2022). After submission of this as a preprint, we learned that their result has also been submitted, but we have not seen it.

We now briefly describe how we prove Theorem \ref{thm:approxSNDTC}. In Section \ref{defcounterexample}, we describe how we pick our counterexample. For those who have read the papers on the Strong Nine Dragon Tree Conjecture, it follows the same set up as usual. We fix positive integers $k$ and $d$,  and we argue that a vertex-minimal counterexample to Theorem \ref{thm:approxSNDTC} for our choice of $k$ and $d$ decomposes into $k+1$ forests where $k$ of them are spanning trees. Let $F$ be the forest which is not spanning (which exists, otherwise the fractional arboricity is too large). We pick our decomposition such that we minimize the number and sizes of large components in $F$ (i.e.\ those with more than $d' := \dPrimeWithoutUsingEll$ edges). We then define a subgraph, which we call the exploration subgraph, from some component $R^*$ of $F$ which is too large, to be the graph we focus on for the paper. We order the components of $F$ in a tree-like structure with $R^*$ being the root. 
The goal now is to push edges away from $R^*$ such that there is enough room to split $R^*$ into smaller components getting the decomposition closer to satisfy the theorem. If there are no possibilities for such operations anymore, we can prove that the density of the exploration subgraph is too large and thus, the fractional arboricity of $G$ is also too large, causing a contradiction.

To do this, we first review the special paths lemma (Lemma \ref{lemma:specialPaths}) in Section \ref{sec:smallComponentsNoSmallChildren}, which was proven in \cite{ndtt}: note that a parent and a child component are always linked by an edge of a spanning tree. If we add this edge to $F$ and the resulting component is not too large, we can take away another edge from a component that is closer to the root in exchange (i.e.\ we move it to one of the spanning trees).

With this in hand, to prove that the exploration subgraph has too large fractional arboricity, it suffices to show that each component of $F$ has few small children. We do this in Section \ref{sec:boundNumSmallChildren}. This is where the advancement lies. Prior to this paper, it has not been clear how to ensure any constant number (assuming $k$ and $d$ are fixed) of small children if $d > 2(k+1)$.
We proceed similar to the approach in \cite{sndtcDLeqKPlusOne}: in most cases it is possible to exchange the spanning-tree edge between a (relatively large) parent $K$ and a (small) child component with an edge of $K$. This splits $K$ into two parts and one of the parts is linked with the small child. If we are lucky, the two new components are smaller than $K$, which might already mean progress. But in the worst case the joined component could be too large. But if $K$ has many small children and we do many of these exchanges in a row such that every part of $K$ is linked with at most one small child, then we can guarantee that none of the new components has more edges than $K$ initially had. After we achieved this, we might need to perform a special path augmentation at a spanning-tree edge to another remaining small child: for very technical reasons the tree structure of the components could have been jumbled up by the previous operations, which the special path augmentation makes up for. If there are at least $\valueOfEllPlusOneInText + 1$ small children connected to $K$ with edges of the same spanning tree, this augmentation technique succeeds. In \cite{sndtcDLeqKPlusOne} the authors implemented this approach for $d \leq 2(k+1)$, where $3$ children have to be considered, and achieved the same edge bound as this paper does, however, using a slightly lengthier argument. 
Finding a general rule with which edge of $K$ a spanning-tree edge to a small child is exchanged and in which order the spanning-tree edges should be exchanged is the main achievement of this paper.
Note that for $d \leq k + 1$ the edge bound is $d$ and thus best possible. We note that the argument in this paper does not prove the Strong Nine Dragon Tree Conjecture when $k+1 < d \leq 2(d+1)$, which was shown in \cite{mies2024strong}. The argument in \cite{mies2024strong} required a reconfiguration argument on components which need not be small in certain cases, and a more complicated counting argument to obtain the optimal bound.

After it has been accomplished that $K$ has at most $\valueOfEllPlusOneInText$ small children generated from a single spanning tree, it is a routine check to argue that the density of the exploration subgraph is too large.

The paper is organized as follows. In Section \ref{defcounterexample} we set up all definitions needed to define the counterexample. In Section \ref{sec:smallComponentsNoSmallChildren}, we review the special paths algorithm, and prove small components do not have small children. In Section \ref{sec:exchangeEdges}, we prove the root component has no small children, and also introduce a useful exchange operation. In Section \ref{sec:boundNumSmallChildren}, we show each component has only few small children. In Section \ref{sec:finish}, we perform the density calculation and finish the proof.

\section{Defining the Counterexample}
\label{defcounterexample}

The goal of this section is to set up everything we need to define  a minimal counterexample to Theorem \ref{thm:approxSNDTC}. First we pin down some basic notation. For a path $P$ with $l$ vertices, we will write $P = [v_1,\ldots,v_l]$ where $v_i v_{i+1}$ is an edge for all $i \in \{1,\ldots,k-1\}$. As we will also consider digraphs, we will use the notation $(u,v)$ to be a directed edge from $u$ to $v$, and we extend the above notation for paths to directed paths if directions are used. 
In a tree $T$ with undirected edges let $\pathIn{T}{x}{y}$ denote the unique path in $T$ from $x$ to $y$.
For a tree $T$ which has its edges directed towards a root vertex $r$, and $x, y \in V(T)$ such that $x$ is a descendant of $y$ in $T$, we let $\pathIn{T}{x}{y}$ be the unique directed path from $x$ to $y$ in $T$.

For the rest of the paper we fix integers $k,d \in \mathbb{N}$ where $k,d \geq 1$. Furthermore, we let $\ell := \valueOfEllInText = \valueOfEllPlusOneInText - 1$ and $d' := \dPrimeWithoutUsingEll = \dPrimeUsingEll$. We will later learn why $\ell$ is an important number. The following fact will be useful throughout the paper:

\begin{obs} \label{obs:twoEllPlusOne}
    $2\ell + 1 \leq d \leq d'$.
\end{obs}
\begin{proof}
    It is clear that $2\ell + 1 \leq d$ holds. For the second inequality, note if $\ell = 0$, we have $d = d'$. If $\ell \geq 1$, we have that $d > k + 1$ and thus, $\frac{1}{2}\valueOfEllPlusOneInText < \frac{d}{k+1}$, which yields $d < d'$.
\end{proof}

We always assume that we have a graph $G$ which is a vertex-minimal counterexample to Theorem \ref{thm:approxSNDTC}. The first observation we need is that $G$ decomposes into $k$ spanning trees and another forest. This fact follows from a minor tweak to the proof of Lemma~2.1 of \cite{ndtt}. Thus, we omit the proof.

\begin{lemma}[Lemma~2.1 \cite{ndtt}] \label{lemma:minimalCounterExample}
    Every graph $G$ that is a vertex-minimal counterexample to Theorem \ref{thm:approxSNDTC} admits a decomposition into forests $T_1, \dots, T_k, F$ such that $T_1, \dots, T_k$ are spanning trees.
\end{lemma}

Note that if $G$ decomposes into $k$ spanning trees and a forest $F$, it follows that $F$ is disconnected, or otherwise, $\fracArb(G) = k+1$, a contradiction.\\
Given a decomposition of $G$, we will want to measure how close it is to satisfying Theorem  \ref{thm:approxSNDTC}. This is captured in the next definition:

\begin{definition}
    The \textit{residue function $\rho(F)$} of a forest $F$ is defined as the tuple $(\rho_{v(G)-1}(F), \\  \rho_{v(G)-2}(F), \ldots, \rho_{d'+1}(F))$, where $\rho_i(F)$ is the number of components of $F$ having $i$ edges.
\end{definition}

We will want to compare residue function values of different forests using lexicographic ordering and are interested in the decomposition with one forest minimizing the residue function.

\begin{notation}
    Over all decompositions into $k$ spanning trees and a forest $F$ we choose one where $F$ minimizes $\rho$ with respect to lexicographic order. We call this minimum tuple $\rho^*$. This forest $F$ has a component $R^*$ containing at least $d' + 1$ edges.
\end{notation}

We will now choose a vertex from $R^*$ which will function as a root for the $k$ spanning trees of $G$. We pick $r$ from ``the center'' of $R^*$, which is formalized in the following lemma.

\begin{lemma} \label{lemma:rExistence}
    The component $R^*$ contains a vertex $r$ such that for every edge $e \in E(R^*)$ that is incident to $r$ we have that the component of $r$ in $R^* - e$ has at least $\ell + 1$ edges.
\end{lemma}
\begin{proof}
    For every $v \in V(R^*)$ let $\beta(v) \in \mathbb N$ be the maximum number such that for every edge $e \in E(R^*)$ that is incident to $v$ we have that the component of $v$ in $R^* - e$ has at least $\beta(v)$ edges.\\
    Suppose towards a contradiction that the lemma is false. Let $r' \in V(R^{*})$ such that $\beta(r')$ is maximized over all vertices. Thus $\beta(r') \leq \ell$. We will show that a neighbour of $r'$ has strictly larger $\beta$.
    Let $r''r' \in E(R^*)$ such that the component of $r'$ in $R^* - r''r'$ has exactly $\beta(r')$ edges.
    We show that $\beta(r'') > \beta(r')$.
    First, note that the component of $r''$ in $R^* - r''r'$ has at least
    \[e(R^*) - (\beta(r')+1) \geq d' + 1 - (\ell + 1) \geq \ell + 1 > \beta(r')\]
    edges by Observation \ref{obs:twoEllPlusOne}.
    Now, let $x \in V(R^*)$ such that $r''x \in E(R^*)$ and $x \neq r'$. Then the component of $r''$ in $R^* - r''x$ has at least $\beta(r') + 1$ edges, contradicting our choice of $r'$. 
\end{proof}

We fix one choice of $r$ with the properties described in Lemma \ref{lemma:rExistence} for the rest of the paper.

\begin{definition}
    We define $\mathcal F$ to be the set of decompositions into forests $(T_1, \dots, T_k, F)$ of $G$ such that $T_1, \dots, T_k$ are directed spanning trees of $G$ where the arcs of $T_1, \ldots T_k$ are directed towards $r$ and $F$ is an undirected forest containing the component $R^*$.\\
    We let $\mathcal F^* \subseteq \mathcal F$ be the set of decompositions $(T_1,\ldots,T_k,F) \in \mathcal F$ such that $\rho(F) = \rho^*$.
\end{definition}

The next definition is simply to make it easier to talk about decompositions in $\mathcal F$.

\begin{definition}
    Let $\mathcal T = (T_1, \dots, T_k, F) \in \mathcal F$. We say that the arcs of $T_1, \dots, T_k$ are \textit{blue} and the (undirected) edges of $F$ are \textit{red}. We define $E(\mathcal T) := E(T_1) \cup \dots \cup E(T_k) \cup E(F)$. 
    Furthermore, we let $\redForest(\mathcal T) := F$ and for any $b \inOneToK$ we let $\blueTree(\mathcal T) := T_b$. For any induced subgraph $H$ of $(V(G), E(\mathcal T))$ we call the connected components of $\redForest(\mathcal T)[V(H)]$ \textit{red components of $H$} and let $e_r(H)$ denote the number of red edges of $H$.
\end{definition}

Finally, we can define the critical subgraph which we will focus on for the rest of the paper:

\begin{definition} \label{def:explSubgraph}
    Let $\mathcal T \in \mathcal F$.
    The \textit{exploration subgraph} $H_{\mathcal T}$ of $\mathcal T$ is the subgraph of the mixed graph $(V(G), \; E(\mathcal T))$ that is induced by the vertex set consisting of all vertices $v$ for which there is a sequence of vertices $r=x_1, \dots, x_l = v$ such that for all $1 \leq i < l$ we have that $(x_i, x_{i+1}) \in \bigcup_{b \inOneToK} E(\blueTree(\mathcal T))$ or 
    $x_i x_{i+1} \in E(\redForest(\mathcal T))$.
\end{definition}

In the next lemma we show that the red subgraph of an exploration subgraph has a natural density bound.

\begin{lemma} \label{lemma:densityRedExplSG}
    \[\frac{e_r(H_{\mathcal T})}{v(H_{\mathcal T}) - 1} \leq \density.\]
\end{lemma}
\begin{proof}
    We have that $H_{\mathcal T}$ has exactly $k(v(H_{\mathcal T}) - 1)$ blue arcs since each vertex $v \in V(H_{\mathcal T}) - r$ has exactly $k$ blue outgoing arcs and $r$ has no blue outgoing arc in $E(\mathcal T)$. We conclude 
    \[
    k + \density
    \geq \fracArb(G)
    \geq \frac{e(H_{\mathcal T})}{v(H_{\mathcal T}) - 1}
    \geq \frac{k(v(H_{\mathcal T}) - 1) + e_r(H_{\mathcal T})}{v(H_{\mathcal T}) - 1}
    = k + \frac{e_r(H_{\mathcal T})}{v(H_{\mathcal T}) - 1}.
    \]
\end{proof}

We will want to focus on red components of $\mathcal H_{\mathcal T}$ with low edge density:

\begin{definition}
    A red component $K$ is \textit{small} if $e(K) \leq \ell$.
\end{definition}

Note that $e(K) \leq \ell$ if and only if $e(K) < \frac{d}{k+1}$, which yields the following observation.

\begin{obs} \label{obs:ellSmallerDensity}
    A red component $K$ is \textit{small} if and only if
    $\frac{e(K)}{v(K)} < \density$.
\end{obs}

Now we turn our focus to the notion of legal orders, which is an ordering of the red components of an exploration subgraph that loosely tells us in what order we should augment the decomposition.

\begin{definition} \label{def:legalOrder}
    Let $\mathcal T \in \mathcal F$ and let $\sigma = (R_1,\ldots,R_t)$ be a sequence of all red components in $H_{\mathcal T}$. We say $\sigma$ is a \textit{legal order} for $\mathcal T$ if $R_1 = R^*$, and further, for each $1 < j \leq t$, there is an $i_j < j$ such that there is a blue arc $(x_j, y_j)$ with $x_j \in V(R_{i_j})$ and $y_j \in V(R_j)$. 
\end{definition}

It will be useful to compare legal orders, and we will again do so using lexicographic ordering.

\begin{definition}
    Let $\mathcal T, \mathcal T' \in \mathcal F$ and suppose that $\sigma = (R_1, \dots, R_t)$ and $\sigma' = (R'_1, \dots, R'_{t'})$ are legal orders for $\mathcal T$ and $\mathcal T'$, respectively. We say $\sigma$ is \textit{smaller than} $\sigma'$, denoted $\sigma < \sigma'$, if $(e(R_1), \dots, e(R_t))$ is lexicographically smaller than $(e(R'_1), \dots, e(R'_{t'}))$. If $t \neq t'$, we extend the shorter sequence with zeros to make the orders comparable.
\end{definition}

To make it easier to discuss legal orders, we introduce some more vocabulary:

\begin{definition}
    Suppose $\sigma = (R_1, \dots, R_t)$ is a legal order for $\mathcal T \in \mathcal F$. If $v \in V(R_j)$ we write $i_\sigma(v) := j$.
    For $U \subseteq V(H_{\mathcal T})$ we define $i_\sigma(U) := \min_{u \in U} i_\sigma(v)$ (in particular, $i_\sigma(\varnothing) = \infty$) and for any subgraph $H \subseteq H_{\mathcal T}$ we let $i_\sigma(H) := i_\sigma(V(H))$.
\end{definition}

For the purposes of tiebreaking how we pick legal orders, we introduce the next graph:

\begin{definition} \label{def:auxTree}
    Let $\mathcal T \in \mathcal F$ and let $\sigma = (R_1,\ldots,R_t)$ be a legal order for $\mathcal T$. 
    Compliant to Definition \ref{def:legalOrder} we choose a blue arc $(x_j, y_j)$ for all $1 < j \leq t$. There might be multiple possibilities for this, but we simply fix one choice for $\mathcal T$ and $\sigma$.
    By removing all the blue arcs from $H_{\mathcal T}$ that are not in $\{(x_j, y_j) \: | \: 1 < j \leq t\}$, we obtain
    the \textit{auxiliary tree of $\mathcal T$ and $\sigma$} and denote it by $Aux(\mathcal T, \sigma)$.
    We always consider $Aux(\mathcal T, \sigma)$ to be rooted at $r$.
\end{definition}

Note that in the blue spanning trees of decompositions of $\mathcal F$ the arcs are directed towards $r$ while in an auxiliary tree blue arcs are directed away from the root $r$.\\
With this, we are in position to define our counterexample. As already outlined, $G$ is a vertex-minimal counterexample to Theorem \ref{thm:approxSNDTC}. Further, we pick a legal order $\sigma^* = (R^*_1, \ldots, R^*_{t^*})$ for a decomposition $\mathcal T^* \in \mathcal F^*$ such that there is no legal order $\sigma$ with $\sigma < \sigma^*$ for any $\mathcal T \in \mathcal F^*$. We will use these notations for the minimal legal order and decomposition throughout the rest of the paper.

Additionally, we define a few terms for $\mathcal T^*$ and $\sigma^*$.
\begin{definition}
    Let $j \in \{2, \ldots, t\}$ and let $x_j$, $y_j$ and $i_j$ be defined as they were fixed in Definitions \ref{def:legalOrder} and \ref{def:auxTree}.
    Let $w(K) := y_j$ and we call $(x_j, y_j)$ the \textit{witnessing arc of $K$}.\\
    Furthermore, we say that $R_{i_j}$ is the parent of $R_j$. On the other hand, we call $R_j$ a child of $R_{i_j}$ (that is generated by $(x_j, y_j)$).
\end{definition}

We now outline how we will show that the fractional arboricity of the counterexample graph is larger than $\density$.
In Lemma \ref{lemma:smallNoSmallChildren} we will show that the parent of a small red component is not small. In Lemma \ref{lemma:numSmallChildren} we will bound the number and sizes of the small children of non-small components. Lemma \ref{lemma:densityKPlusChildren} will show that this bound is enough such that on average the density of a parent component and all its small children is at least $\density$. Finally, we will show that this leads to a contradiction to Lemma \ref{lemma:densityRedExplSG}.

\section{Small Components Do Not Have Small Children} \label{sec:smallComponentsNoSmallChildren}

In this section we consider the first method to find a smaller legal order or shrink a component with more than $d'$ edges. This method from \cite{ndtt} roughly works the following way:
if a blue arc $e$ connecting two red components can be coloured red without increasing the residue function, then in certain cases we can find a red edge $e'$ that can be coloured blue in exchange. In order to find this edge we need to look for a certain blue directed path that ends at $e$ and starts at a vertex $v_0$ and $v_0$ has to be closer to $R^*$ with respect to the legal order than the tail of $e$. First, we formalize the requirements for such a blue path:

\begin{definition}
    Let $\sigma = (R_1, \dots, R_t)$ be a legal order for $\mathcal T \in \mathcal F$. We call a blue directed path $P=[v_0, v_1, \dots, v_l] \subseteq \bigcup_{b \inOneToK} \blueTree(\mathcal T)$ that is contained in $H_{\mathcal T}$ \textit{special with respect to $\mathcal T$, $\sigma$ and $(v_{l-1}, v_l)$} if $v_{l-1}$ and $v_l$ are in different components of $\redForest(\mathcal T)$, $i_\sigma(v_l) > i_\sigma(v_0)$ and furthermore, $v_0$ needs to be an ancestor of $v_{l-1}$ in $Aux(\mathcal T, \sigma)$ if both of them are in the same component of $\redForest(\mathcal T)$. \\
    For two special paths $P=[v_0, v_1, \dots, v_l]$ and $P'=[v'_0, v'_1, \dots, v'_{l'}]$ with respect to $\mathcal T$, $\sigma$ and $(v_{l-1}, v_l)$ we write $P \leq P'$ if $i_\sigma(v_0) < i_\sigma(v'_0)$, or if $i_\sigma(v_0) = i_\sigma(v'_0)$ and $v_0$ in $Aux(\mathcal T, \sigma)$ is an ancestor of $v'_0$. We call a special path $P$ with respect to $\mathcal T$, $\sigma$ and $(x, y)$ \textit{minimal} if there is no special path $P' \neq P$ with respect to $\mathcal T$, $\sigma$ and $(x, y)$ such that $P' \leq P$.
\end{definition}
Note that for every special path $P'$ with respect to $\mathcal T$, $\sigma$ and $(x, y)$ there exists a minimal special path $P$ with respect to $\mathcal T$, $\sigma$ and $(x, y)$ such that $P \leq P'$. Furthermore, note that if we have a minimal special path $P=[v_0, v_1, \dots, v_l]$ with respect to $\mathcal T$, $\sigma$ and $(v_{l-1}, v_l)$, we have $v_0 \neq r$ because $r$ has no outgoing blue arc by construction. Therefore, $v_0$ has a parent vertex in $Aux(\mathcal T, \sigma)$, which we denote by $v_{-1}$. Note that the edge $v_{-1} v_0$ is red because of the minimality of $P$ and since all blue arcs in $Aux(\mathcal T, \sigma)$ are directed away from $r$ in the auxiliary tree.\\
The following lemma describes which modifications to the decomposition can be made if a minimal special path exists and how they change the legal order. An illustration of these modifications can be seen in Figure \ref{fig:specialPaths}.

\begin{lemma}[cf.\ Lemma 2.4 and Corollary 2.5 in \cite{ndtt}]	\label{lemma:specialPaths}
    Let $\sigma = (R_1, \dots, R_t)$ be a legal order for $\mathcal T = (T_1, \ldots, T_k, F) \in \mathcal F$.
    Furthermore, let $P = [v_0, v_1, \dots, v_l]$ be a minimal special path with respect to $\mathcal T$, $\sigma$ and $(v_{l-1},v_l)$. Let $i_0 := i_\sigma(v_0)$.\\
    Then there is a partition into forests $\mathcal T' = (T'_1, \dots, T'_k, F')$ of $G$ such that $T'_1, \dots, T'_k$ are spanning trees rooted at $r$ whose arcs are directed to the respective parent vertex, forest $F'$ exclusively consists of undirected edges and if $R'$ denotes the component in $F'$ containing $r$, we have that
    \begin{enumerate}
        \item $F' = \big( F + v_{l-1} v_l \big) - v_{-1}v_0$.
        \item $(v_0, v_{-1}) \in \bigcup_{b=1}^k E(T'_b)$.
        \item $\big\{(x, y) \in E(T'_b) \: | \: i_\sigma(x) < i_0\big\} 
        = \big\{(x, y) \in E(T_b) \: | \: i_\sigma(x) < i_0\big\}$ for all $b \inOneToK$.
        \item $\begin{aligned}[t]
            \bigcup_{b=1}^k \big\{uv \: | \: (u, v) \in E(T'_b)\big\} 
            = \bigg(\Big(\bigcup_{b=1}^k \big\{uv \: | \: (u, v) \in E(T_b) \big\}\Big) - v_{l-1}v_l \bigg) + v_0v_{-1}.
        \end{aligned}$
        \item If $i_0 > 1$, then $R^* = R'$, $\mathcal T' \in \mathcal F$ and there exists a legal order $\sigma' = (R'_1, \dots, R'_{t'})$ for $\mathcal T'$ with $R'_j = R_j$ for all $j < i_0$ and $e(R'_{i_0}) < e(R_{i_0})$, where $R'_{i_0}$ is the component of $v_{-1}$ in $F'$. Thus, $\sigma' < \sigma$.
    \end{enumerate}
\end{lemma}

\begin{figure}[!htb]
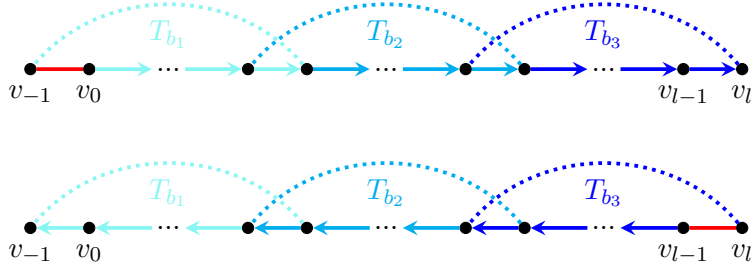

    \center 
    \specialPaths
    \center
    \specialPathsAugm
    \caption{
        An example of an augmentation of a minimal special path that consists of three segments of three different spanning trees. A blue dotted path is the unique path in its spanning tree between the two vertices when orientations are ignored. Note that the minimal special path can be chosen such that the dotted paths are edge-disjoint to the special path.
    }
    \label{fig:specialPaths}
\end{figure}

Since we want to find smaller orders than $\sigma^*$ in some of the following proofs to arrive at a contradiction, it is desirable that the fifth point holds in an application of Lemma~\ref{lemma:specialPaths}. In the event this is not the case, we can still gain more structure. We want to show this more formally with the next lemma:

\begin{lemma}	\label{lemma:v0NotInR}
    If Lemma \ref{lemma:specialPaths} is applicable such that $\rho(F + v_{l-1} v_l) = \rho^*$ holds, then $e(R_{i_0}) \leq d'$ and therefore $i_0 > 1$ such that 5.\ can be applied. Moreover, in this case we have $\mathcal T' \in \mathcal F^*$ for the partition obtained.
\end{lemma}
\begin{proof}
    We use the notation of Lemma~\ref{lemma:specialPaths}. We have that $R_{i_0}$ is split into two strictly smaller components in $F'$. If $\rho(F + xy) = \rho^*$ and $e(R_{i_0}) > d'$, then we had $\rho(F') < \rho^*$, a contradiction.
\end{proof}

The case treated in Lemma \ref{lemma:v0NotInR} still yields a contradiction since 5. of Lemma \ref{lemma:specialPaths} provides a smaller legal order than $\sigma^*$. This fact is captured in the following corollary.

\begin{corollary} \label{cor:specialPathInStandardCase}
    Let $\sigma = (R_1, \ldots, R_t)$ be a legal order for $\mathcal T \in \mathcal F$.
    If there is a special path $P = [v_0, \ldots, v_l]$ with respect to $\mathcal T$, $\sigma$ and $(v_{l-1}, v_l)$ such that $e(R_i) \leq e(R^*_i)$ for all $i \in \{1, \ldots, i_\sigma(v_0)\}$, then $\rho(\redForest(\mathcal T) + v_{l-1} v_l) > \rho^*$.
\end{corollary}

Corollary \ref{cor:specialPathInStandardCase} shows how we will make use of special paths later. After doing some exchanges between the forests of $\mathcal T^*$ we will obtain a decomposition $\mathcal T$ and a special path as described in the corollary, but we will also have that $\rho(\redForest(\mathcal T) + v_{l-1} v_l) = \rho^*$.
Next we want to look at what consequences Lemma \ref{lemma:specialPaths} has on the relation of child and parent components:

\begin{lemma}[Corollary 2.5 from \cite{ndtt}] \label{lemma:sumOfChildRelation} \phantom{bla} \\ 
    Let $C$ be a child of $K$ that is generated by $(x, y)$. Then $e(K) + e(C) \geq d'$. 
\end{lemma}
\begin{proof}
    Suppose to the contrary that $e(K) + e(C) < d'$. Since $i_{\sigma^*}(y) > i_{\sigma^*}(x)$, we have that $[x, y]$ is a special path with respect to $\mathcal T^*$, $\sigma^*$ and $(x, y)$. Furthermore, $\rho(\redForest(\mathcal T^*) + xy) = \rho^*$, but this contradicts Corollary \ref{cor:specialPathInStandardCase}.
\end{proof}

\begin{lemma} \label{lemma:smallNoSmallChildren}
    Let $K$ be a small red component of $\explSG$. Then $K$ does not have small children and thus, every small red component of $\explSG$ has a parent component that is not small.
\end{lemma}
\begin{proof}
    Let $C$ be a small red component of $\explSG$. If $C$ does not have a parent component, then $C = R^*$, which is a contradiction since $e(R^*) \geq d' + 1 > \ell$ by Observation \ref{obs:ellSmallerDensity}. Thus, let $K$ be the parent of $C$. By Lemma \ref{lemma:sumOfChildRelation}, we have that $e(K) \geq d' - e(C) \geq d ' - \ell \geq \ell + 1$ by Observation \ref{obs:ellSmallerDensity} and hence, $K$ is not small.
\end{proof}

\section{\boldmath\texorpdfstring{$R^*$}{R*} Does Not Have Small Children} \label{sec:exchangeEdges}

In this section, we define a useful exchange operation and show how to reorient edges after the exchange to maintain the proper structure of the decomposition. This operation has already been used in previous Nine Dragon Tree papers, starting with \cite{ndtt}. We then use this exchange operation to show that $R^*$ does not have small children. For this section we define $\mathcal T \in \mathcal F$ together with a legal order $\sigma$ of $\mathcal T$.

\begin{definition} \label{def:edgeExchange}
    Let $e \in E(\blueTree(\mathcal T))$ for some $b \inOneToK$ and $e' \in E(\redForest(\mathcal T))$. If $(\blueTree(\mathcal T) - e) + e'$ is a spanning tree and $(F - e') + e$ is a forest (ignoring orientations), we say that $e'$ can be exchanged with $e$, and say that $e \leftrightarrow e'$ holds in $\mathcal T$.
\end{definition}

The next lemma is obvious and we omit the proof, but it usefully characterizes when $e \leftrightarrow e'$. However, we refer the reader to Figure \ref{fig:swapDef} for an illustration.

\begin{lemma} \label{lemma:exchange}
    Let $u \in V(G) - r$, $u'$ be the parent vertex of $u$ in $\blueTree(\mathcal T)$ for some $b \inOneToK$ and $e = v_1 v_2 \in E(\redForest(\mathcal T))$ such that $(\redForest(\mathcal T) + uu') - e$ does not contain a cycle.\\
    Then, the following are equivalent:
    \begin{enumerate}[(a)]
        \item $(u, u') \leftrightarrow e$ holds in $\mathcal T$.
        \item The edge $(u,u')$ lies in the unique cycle (ignoring orientations) of $\blueTree(\mathcal T) + e$.
        \item Up to relabelling $v_1$ as $v_2$, $v_1$ is a descendant of $u$ in $\blueTree(\mathcal T)$ and $v_2$ is not.
    \end{enumerate}
    Furthermore, if these conditions are met, then after exchanging $(u, u')$ and $e$ between $\blueTree(\mathcal T)$ and $\redForest(\mathcal T)$, orienting $e$ towards $v_2$, removing the orientation of $e$ and reorienting the path $\pathInBlueTree{\mathcal T}{v_1}{u}$, the resulting decomposition is again in $\mathcal F$, and we say we obtain the resulting decomposition from $\mathcal T$ by performing $(u, u') \leftrightarrow e$.
\end{lemma}

\doubleFigureCentered{
    \swapDefBefore
}{
    \swapDefAfter
}{
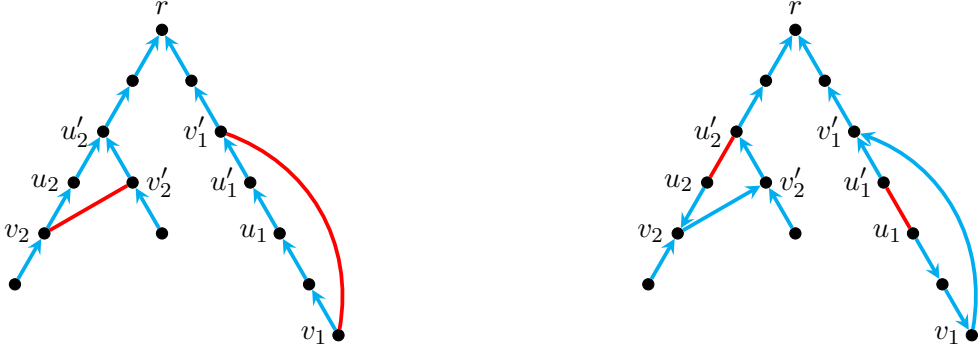
    An example where we have $(u_1, u'_1) \leftrightarrow v_1 v'_1$ and $(u_2, u'_2) \leftrightarrow v_2 v'_2$.
}{
    \label{fig:swapDef}
}

\begin{lemma} \label{lemma:rootNoChildren}
    The component $R^*$ does not have small children.
\end{lemma}
\begin{proof}
    Suppose to the contrary that $R^*$ has a small child $C$ generated by $(x, x') \in E(\blueTree(\mathcal T^*))$ for some $b \inOneToK$.
    Let $P = [x_1, \ldots, x_n]$ be the path from $x$ to $r$ in $\redForest(\mathcal T^*)$. Let $i \in \{1, \ldots, n\}$ such that $x_i$ is the first vertex on $P$ that is not a descendant of $x$ in $\blueTree(\mathcal T^*)$. This vertex exists and $i > 1$ since $x$ is a descendant of $x$ in $\blueTree(\mathcal T^*)$ and $r$ is not. We obtain $\mathcal T$ from $\mathcal T^*$ by performing $(x, x') \leftrightarrow x_{i-1} x_i$. By Lemma \ref{lemma:exchange} we have that $\mathcal T \in \mathcal F$.
    The component $K_r$ of $r$ in $\redForest(\mathcal T)$ contains at least $\ell + 1$ edges by the definition of $r$. For the component $K_x$ of $x$ in $\redForest(\mathcal T)$ we have 
    \[
    e(K_x) 
    \leq e(R^*) - |\{x_{i-1} x_i\}| - e(K_r) + |\{xx'\}| + e(C)
    < e(R^*).
    \]
    Since $K_r$ is a proper subgraph of $R^*$, we have that $\rho(\redForest(\mathcal T)) < \rho^*$, which is a contradiction.
\end{proof}

\section{Non-small Components Do Not Have Many Small Children} \label{sec:boundNumSmallChildren}

In this section we fix $b \inOneToK$ and a non-small red component $K \neq R^*$ of $\explSG$. Our goal is to show that $K$ has at most $\ell + 1$ small children generated by $\blueTree(\mathcal T^*)$. 
In fact, we will show a slightly stronger statement, which is Lemma \ref{lemma:numSmallChildren}.
In order to show this, we combine the two strategies of the previous two sections. Let us give some intuition for the approach before delving into the technical details. First, we perform edge exchanges as defined in Definition \ref{def:edgeExchange}. 
These exchanges will cut off a part of $K$ and link it with a small child. After we have done this a sufficient number of times, every part will be linked to at most one child, which will guarantee us that no part will have more $e(K)$ edges. If the exchanges have not messed up the legal order before $K$, then we will have arranged that the component of $w(K)$ has less than $e(K)$ edges, which immediately gives us a smaller legal order than $\sigma^*$. If the legal order is jumbled by the exchanges, we will have arranged that there is an spanning-tree arc to a small child of $K$, which we have not exchanged and which is the end of a minimal special path whose start is so far back in the legal order where it is still equal to $\sigma^*$. If we augment this minimal special path we will retrieve a smaller legal order than $\sigma^*$ as we previously did. We will later call the second to last vertex of this path (which is in $K$ and has an arc to a small child) a special vertex. More specifically, each intermediate decomposition will have a such a vertex, even if we are not aiming to perform a special path augmentation yet, that we call a special vertex (however, it might change after an exchange operation and pass its ``special'' properties to another vertex).

Another challenge lies in not accidentally reorienting generating edges while performing edge exchanges. Since a blue directed path ending at a ``generating vertex'' (i.e.\ a tail of a generating edge) is reoriented by the exchange operation introduced in Lemma \ref{lemma:exchange}, we aim to make an edge exchange with a generating vertex that does not have another generating vertex as a descendant. In Lemma \ref{lemma:mainAugmentation} this vertex will be called $x$ or it will be the special vertex $v_*$. However, for technical reasons we have to take another vertex (which will be called $y$) for the exchange, which also has the nice property of not reorienting generating edges. When we are forced to perform an edge exchange with $v_*$, we will arrange that after the exchange there is a blue path from $v_*$ to another generating vertex, which will then meet all the criteria for being a special vertex in the resulting decomposition.
When it comes to not reorienting generating edges, we will actually make an exception: during the phase of edge exchanges operations our goal is to decrease the size of the largest component that stems from $K$. In the worst case we can only split one edge from it (i.e.\ the edge that is exchanged with a generating edge). Thus, if we split off a generating vertex that is not $v_*$ from the largest component together with at least another vertex, then we can allow ourselves to not consider this vertex for future operations, since more progress than expected has already been made, and thus, we may also reorient it. This is covered by 5) in the first subsection.\\
In the first subsection of this section we describe which properties intermediate decompositions before the minimal special path augmentation should have.
Of course, these properties also hold for $\mathcal T^*$.
In the second subsection we then fully describe and prove the augmentation steps we just encountered.

\subsection{Definition of Valid Intermediate States}

\begin{notation} \label{notation:primeIsParent}
   Fix a positive integer $\ell' \leq \ell$. By $\X$ we denote the set of vertices $x$ such that there is a small child of $K$ having at most $\ell'$ edges that is generated by $(x, x') \in E(\blueTree(\mathcal T^*))$.
    If not described otherwise, $x'$ denotes the parent of $x \in \X$ in $\blueTree(\mathcal T^*)$.
    Let $\mathcal C(K)$ denote the set of small children of $K$. For $\ell' \inZeroToEll$ let $\mathcalC$ denote the number of small children of $K$ having at most $\ell'$ edges.
\end{notation}
The following definition will help us describing which vertices may be reoriented by performing edge exchanges. 
This might not be clear straightaway, but we will give an intuition afterwards.

\begin{definition} \label{def:subtreeOfGeneratingVertex}
    Let $\mathcal T \in \mathcal F$ and $U \subseteq V(K)$. Furthermore, for $v \in V(G)$ let $\Delta(v)$ denote the set of descendants of $v$ in $\blueTree(\mathcal T)$. For every $u \in U$ let $T_u(U, \mathcal T)$ be the subtree of $\blueTree(\mathcal T)$ with root $u$ and vertex set $\Delta(u) \setminus \bigParan{\bigcup_{u' \in U \cap \Delta(u)} \Delta(u')}$.
    Let $\I := \min_{x \in \X} \iSigStar(T_x(\X, \mathcal T^*))$.
\end{definition}

Note that in Definition \ref{def:subtreeOfGeneratingVertex} we have that $V(T_u(U, \mathcal T))$ and $V(T_{u'}(U, \mathcal T))$ are disjoint for any distinct vertices $u, u' \in U$. Later, we will use this definition with $U \subseteq \X$ and our setup will guarantee that when performing $(u, u') \leftrightarrow e$, where $u \in U$, then only blue arcs of $T_u(U, \mathcal T)$ will be reoriented and in particular, no other ``child-generating'' arcs will be reoriented whose tails are in $U$.\\
Let $\ell' \inZeroToEll$, $\mathcal T \in \mathcal F$ and $a \in \X$. In the rest of this subsection we enumerate nine conditions $1), \ldots, 9)$ which all need to hold such that we may call $(\mathcal T, a)$ a valid intermediate state for $K$ and $\ell'$. For better readability we will sometimes use the notation $\mathcal T_a := (\mathcal T, a)$.

Let $\mathcal L(\mathcal T)$ be the set of components of $\redForest(\mathcal T)[V(K)]$ and for every $L \in \mathcal L(\mathcal T)$ let $\bar L$ be the component of $\redForest(\mathcal T)$ containing $L$. 
\begin{enumerate}[1)]
    \item All components of $\redForest(\mathcal T^*)$ that are not from $\mathcalC \cup \{K\}$ are also components of $\redForest(\mathcal T)$, $\bigcup_{C \in \mathcalC} E(C) \subseteq E(\redForest(\mathcal T))$ and the remaining edges of $\redForest(\mathcal T)$ form a subset of $E(K) \cup \generatingEdgesInRed$.
    \item[2a)] If $\ISmallerISigK$, then there is a legal order $(R_1, \ldots, R_t)$ for $\mathcal T$ such that $R_i = R^*_i$ for all $i \in \{0, \ldots, \I\}$.
    \item[2b)] If $\IEqISigK$, then there is a legal order $(R_1, \ldots, R_t)$ for $\mathcal T$ such that $R_i = R^*_i$ for all $i \in \{0, \ldots, \iSigStar(K) - 1\}$ and the witnessing arc of $K$ also is a blue arc in the same orientation as in $\mathcal T^*$ in $E(\mathcal T)$.
    \addtocounter{enumi}{1}
    \item $\min_{x \in \X} \iSigStar(T_x(\X, \mathcal T)) \geq \I$.
    \item For all $L \in \mathcal L(\mathcal T)$, $\bar L$  contains at most one edge of $\generatingEdgesInRed$.
\end{enumerate}
Note that at most one component of $\mathcal L(\mathcal T)$ has at least $\maxOfDAndEOfK - \ell'$ edges (or otherwise 
$e(K) \geq 2(\maxOfDAndEOfK - \ell') + 1 > \maxOfDAndEOfK$ by Observation \ref{obs:twoEllPlusOne}). We denote this component by $K'(\mathcal T)$ if it exists and otherwise we let $K'(\mathcal T)$ be the subgraph of $K$ not containing any vertices, denoted by $\varnothing$. 
\begin{enumerate}[resume*]
    \item For all $x \in \X \cap V(K'(\mathcal T))$, we have that $(x', x) \notin E(\blueTree(\mathcal T))$.
\end{enumerate}
Let $S(\mathcal T)$ be the set of vertices $s$ such that $s \in \X \cap V(K'(\mathcal T))$ and $(s, s') \in E(\blueTree(\mathcal T))$.\\

In the following we will define special vertices. They are linked to special paths in the following way: our augmentations will finally resolve in a valid intermediate state, which might have a special vertex $v_*$. If so, we will do a special path augmentation in the proof of Lemma \ref{lemma:existenceValidColouringForAugment} where the minimal special path ends with $(v_*, v'_*)$. By the definition of special vertices this minimal special path will start relatively close to $R^*$ with respect to $\sigma^*$. As all our other changes on the decomposition will be later in the legal order, the special path augmentation will provide a smaller legal order than $\sigma^*$, which will be a contradiction.

\begin{definition}
    Let $v_* \in \X$ and define $\bar S_{v_*} := S(\mathcal T) + v_*$.
    \begin{itemize}
    \item{If $\ISmallerISigK$, we call $v_* \in \X$ \textit{special for $\mathcal T$} if $(v_*, v'_*) \in E(\blueTree(\mathcal T))$ and 
    $\iSigStar(T_{v_*}(\bar S_{v_*}, \mathcal T)) = \I$.}
    \item{If $\IEqISigK$, we call $v_* \in \X$ \textit{special for $\mathcal T$} if $(v_*, v'_*) \in E(\blueTree(\mathcal T))$ and $w(K) \in V(T_{v_*}(\bar S_{v_*}, \mathcal T))$.}
    \end{itemize}
\end{definition}

\begin{enumerate}[resume*]
    \item[6a)] If there is a special vertex for $\mathcal T$, then $a$ is a special vertex. In this case let $v_* := a$.
    \item[6b)] If there is no special vertex for $\mathcal T$, then $\IEqISigK$ and $a = w(K)$. 
    \addtocounter{enumi}{1}
    \item Let $L \in \mathcal L(\mathcal T)$ contain $a$. Then $\bar L$ does not contain an edge of $\generatingEdgesInRed$ and thus, $\bar L = L$.
\end{enumerate}
Note that some of the notations defined in the following paragraphs are illustrated in Figure \ref{fig:defIntValidStates}.
If $a$ is a special vertex for $\mathcal T$, let $\bar S(\mathcal T_a) := S(\mathcal T) + v_*$ and $\mathring S(\mathcal T_a) := S(\mathcal T) - v_*$.\\
If there is no special vertex for $\mathcal T$, let $\bar S(\mathcal T_a) := \mathring S(\mathcal T_a) := S(\mathcal T)$.\\
If there is a vertex $t \in \X \cap V(K'(\mathcal T))$ such that $tt' \in E(\redForest(\mathcal T))$, then we say that \textit{$t$ extends $K'(\mathcal T)$ in $\mathcal T$}. Note that if 4) and 7) hold for $\mathcal T_a$, then there can be at most one such vertex and this vertex is not $a$. If there is a vertex extending $K'(\mathcal T)$ in $\mathcal T$, we let $u(\mathcal T_a)$ be this vertex, and otherwise we let $u(\mathcal T_a) := a$. Note that if 4) and 5) hold for $\mathcal T_a$, $u(\mathcal T_a)$ always exists and is unique. 
\begin{enumerate}[resume*]
    \item If $K'(\mathcal T) \neq \varnothing$, then $u(\mathcal T_a) \in V(K'(\mathcal T))$.
    \item If $u(\mathcal T_a)$ extends $K'(\mathcal T)$ in $\mathcal T$, then $\pathInBlueTree{\mathcal T}{u(\mathcal T_a)}{r}$ either does not contain a vertex of $\bar S(\mathcal T_a)$ or the first vertex on this path which is in $\bar S(\mathcal T_a)$ is $v_*$.
\end{enumerate}

\singleFigure{\defIntValidStates}{
    A part of a valid intermediate state $\mathcal T_a$ where only the changes to components involving $K$ are shown, such that we have the following properties:
    \begin{itemize}
        \item $V(K) = V(K'(\mathcal T)) \; \dot\cup \; V(\tilde K)$.
        \item $X_2(K) = \{x_1, \ldots, x_5, u(\mathcal T_a), v_*\}$.
        \item $S(\mathcal T_a) = \{x_1, x_2, x_3, u(\mathcal T_a)\} = \mathring S(\mathcal T_a)$.
        \item $\bar S(\mathcal T_a) = \{x_1, x_2, x_3, u(\mathcal T_a), v_*\}$.
        \item $u(\mathcal T_a)$ extends $K'(\mathcal T)$ in $\mathcal T$.
    \end{itemize}
}{\label{fig:defIntValidStates}}

If conditions $1), \ldots, 9)$ are met, we call $\mathcal T_a$ a \textit{valid intermediate state for $K$ and $\ell'$} and denote the set of all valid intermediate states for $K$ and $\ell'$ by $\setOfValidDecomps$. It is routine to check that $\setOfValidDecomps$ is not empty: we state this as an observation.
\begin{obs} \label{obs:TStarValidDecomp}
    If $\ISmallerISigK$, let $a^* \in \X$ be a vertex minimizing $\iSigStar(T_x(\X, \mathcal T^*))$. If $\IEqISigK$, let $a^*$ be a special vertex for $\mathcal T^*$ if one exists, and otherwise let $a^* = w(K)$.\\
    Then $(\mathcal T^*, a^*) \in \setOfValidDecomps$ and thus, $\setOfValidDecomps \neq \varnothing$.
\end{obs}
In the rest of the section, if we say $i)$ holds for $(\mathcal T, a)$, $i)$ always refers to the respective item in this subsection. Note that for $i \in \{1, \ldots, 5\}$ we might just say that $i)$ holds for $\mathcal T$ since in these cases,  $i)$ does not depend on $a$.

\subsection{Augmenting Valid Intermediate States}

In Lemma \ref{lemma:mainAugmentation} we will show that we can perform an edge exchange that gives us a ``better'' valid intermediate state. Before we turn to this lemma, we will show that after any obvious edge exchange, 1), 2a), 2b) and 3) will still hold in the new decomposition.

\begin{lemma} \label{lemma:legalOrderInValidDecomp}
    Let $\ell \inZeroToEll$, $(\mathcal T, a) \in \setOfValidDecomps$ and $x \in \bar S(\mathcal T_a)$. Furthermore, let $v$ be a descendant of $x$ in $\blueTree(\mathcal T)$ and $v'$ not be a descendant of $x$ such that $vv' \in E(K) \cap E(\redForest(\mathcal T))$. Then 1), 2a), 2b) and 3) hold for the decomposition $\mathcal T'$ that can be obtained from $\mathcal T$ by performing $(x,x') \leftrightarrow vv'$.
\end{lemma}
\begin{proof}
    It is clear that 1) holds for $\mathcal T'$.
    Now we prove 2a) and 2b) hold for $\mathcal T'$. First observe that as 3) holds for $\mathcal T_a$ we have that $\iSigStar(\pathInBlueTree{\mathcal T}{v}{x}) \geq \I$.
    First, let $\ISmallerISigK$. 
    Then $\iSigStar(K \cup \bigcup_{C \in \mathcalC} C) > \I$ and thus, there is a legal order $(R'_1, \ldots, R'_{t'})$ for $\mathcal T'$ such that $R'_i = R_i = R^*_i$ for all $i \in \{0, \ldots, \I\}$ and hence, 2a) holds for $\mathcal T'$.\\
    If $\IEqISigK$, then $\iSigStar(K \cup \bigcup_{C \in \mathcalC} C) \geq \iSigStar(K)$ and thus, the witnessing arc of $K$ also exists in $E(\mathcal T')$ and there is a legal order $(R'_1, \ldots, R'_{t'})$ for $\mathcal T'$ such that $R'_i = R_i = R^*_i$ for all $i \in \{0, \ldots, \iSigStar(K) - 1\}$. Hence, 2b) holds for $\mathcal T'$.

    Finally, suppose 3) does not hold for $\mathcal T'$. Thus there is a vertex $z$ with $\iSigStar(z) < \I$ that has path in $\blueTree(\mathcal T')$ to some vertex of $\X$, but it does not have such a path in $\blueTree(\mathcal T)$. Let $(u, u')$ be the first arc on this path that is not in $\blueTree(\mathcal T')$. Then $u \in V(\pathInBlueTree{\mathcal T}{v}{x})$ and thus, $z$ has a path to $x$ in $\blueTree(\mathcal T)$, which is a contradiction. 
\end{proof}

\begin{lemma} \label{lemma:oldLastItem}
    Let $\ell' \inZeroToEll$ and $\mathcal T_a \in \setOfValidDecomps$. Then $\pathInBlueTree{\mathcal T}{u(\mathcal T_a)}{r}$ either does not contain a vertex of $\bar S(\mathcal T_a)$ or $a$ is special and the first vertex on this path which is in $\bar S(\mathcal T_a)$ is $a$.
\end{lemma}
\begin{proof}
    The lemma trivially holds if $u(\mathcal T_a) = a$ is special. If $u(\mathcal T_a)$ extends $K'(\mathcal T)$ in $\mathcal T$, then the lemma holds since 9) holds for $\mathcal T$. If $u(\mathcal T_a) = a = w(K)$ is not special, then $\IEqISigK$ and there is no special vertex for $\mathcal T$ by 6a). Thus, $w(K)$ is not a descendant of any vertex of $\bar S(\mathcal T_a)$.
\end{proof}

\begin{lemma} \label{lemma:mainAugmentation}
    Let $\ell' \inZeroToEll$ and $\mathcal T_a \in \setOfValidDecomps$ such that $\mathring S(\mathcal T_a) \neq \varnothing$ and $K'(\mathcal T) \neq \varnothing$. Then there is a tuple $(\mathcal T', a') \in \setOfValidDecomps$ such that $K'(\mathcal T') \subsetneq K'(\mathcal T)$.
\end{lemma}
\begin{proof}
    If $a$ is special for $\mathcal T$, then let $a =: v_*$. 
    Consider the tree  $Q$ with vertex set $\bar S(\mathcal T_a) + r$ where two vertices $v_1, v_2$ are linked by an arc $(v_1, v_2)$ whenever $v_1$ is a descendant of $v_2$ in $\blueTree(\mathcal T)$. Observe that $Q$ has at least  two vertices which are not $a$ since $\mathring S(\mathcal T_a) \neq \varnothing$.  Let $Q$ be rooted at $r$ such that every vertex has a directed path to $r$. 
    If $a$ is special for $\mathcal T$ and the only non-root leaf in $Q$, then let $x$ be the parent of $a$ (in this case $Q$ is a path from $a$ to $r$). Otherwise let $x$ be a non-root leaf of $Q$. 
    Note that $x \notin \{u(\mathcal T_a), r\}$ and furthermore, if $a$ is special we have $x \neq v_*$.
    Since 8) holds for $\mathcal T_a$, we have that $u(\mathcal T_a) \in K'(\mathcal T)$, and as $x \in \mathring S(\mathcal T_a)$, it follows that both $x$ and $u(\mathcal T_a)$ are in $K'(\mathcal T)$.\\
    Let $\pathIn{K'(\mathcal T)}{u(\mathcal T_a)}{x} = [v_1, \ldots, v_n]$. For $i \in \{2, \ldots, n\}$ let $L_i$ and $L'_i$ be the components of $K'(\mathcal T) - v_{i-1} v_i$ containing $v_i$ and $v_{i-1}$, respectively.
    Note that the notations defined in the following are illustrated on the right side of Figure \ref{fig:subcaseOneOne}.
    Let $i$ be minimal with the property that $\pathInBlueTree{\mathcal T}{v_i}{r}$ 
    visits a vertex of $V(L_i) \cap \mathring S(\mathcal T_a)$ before it reaches any vertex of $\{v_*, r\}$. 
    Note that $i$ exists since $x \in V(L_n) \cap \mathring S(\mathcal T_a)$ and we have $v_i \neq u(\mathcal T_a)$ by Lemma \ref{lemma:oldLastItem} and thus, $i \geq 2$.
    Let $v := v_i$ and $v' := v_{i-1}$ and note that $\pathInBlueTree{\mathcal T}{v'}{r}$ either does not contain a vertex of $\bar S(\mathcal T_a)$ or the first vertex on this path which is in $\bar S(\mathcal T_a)$ is $v_*$ (if $i = 2$ and thus, $v' = u(\mathcal T_a)$, this is true by Lemma \ref{lemma:oldLastItem}). Let $L := L_i$ and $L' := L'_i$. We split into cases depending on how many edges $L'$ has.

    \doubleFigure{\subcaseOneOneBefore}{\subcaseOneOneAfter}{
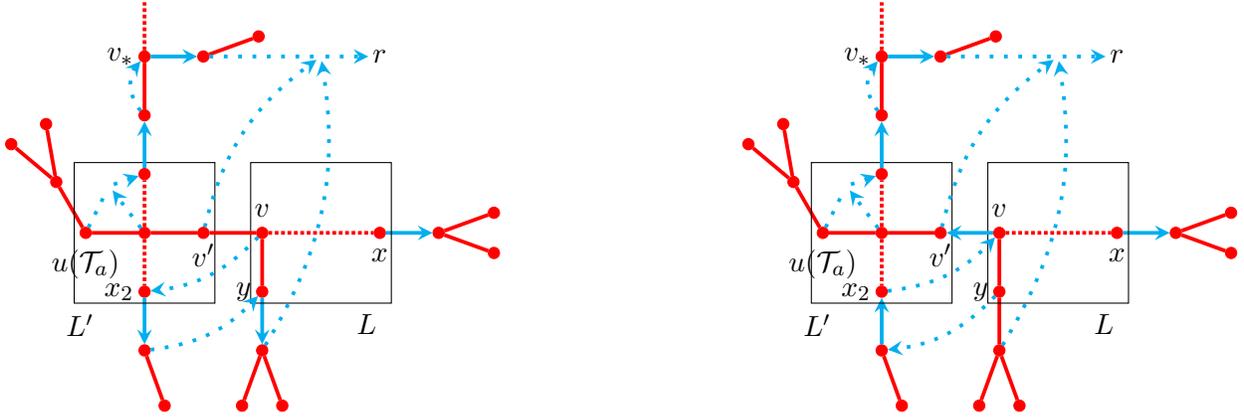
        $\mathcal T_a$ (left side) and $\mathcal T'_a$ (right side) in Subcase 1.1 in the proof of Lemma \ref{lemma:mainAugmentation}.
    }{\label{fig:subcaseOneOne}}
    
    \textbf{Case 1: $e(L') \leq \edgeLimit$:}\\
    Let $y$ be the first vertex on $\pathInBlueTree{\mathcal T}{v}{r}$ that is in $V(L) \cap \mathring S(\mathcal T_a)$.
    
    \textbf{Subcase 1.1: $v'$ is not a descendant of $y$ in $\blueTree(\mathcal T)$:}\\
    Note that $v'$ is a neighbour of $v$ in the red path from $u(\mathcal T_a)$ to $v$ while, compliant to Notation \ref{notation:primeIsParent}, $y'$ shall denote the parent of $y$ in $\blueTree(\mathcal T^*)$.
    We obtain $\mathcal T'$ from $\mathcal T$ by performing $(y, y') \leftrightarrow vv'$ (which reorients $\pathInBlueTree{\mathcal T}{v}{y}$)
    and show that $\mathcal T'_a := (\mathcal T', a) \in \setOfValidDecomps$ and $\mathcal T'_a$ satisfies the lemma. Note that $\mathcal T'_a$ is depicted on the left side of Figure \ref{fig:subcaseOneOne}.\\
    Note that 1), 2a), 2b) and 3) hold for $\mathcal T'$ by Lemma \ref{lemma:legalOrderInValidDecomp}. Hence, we focus on the remaining conditions.\\
    To see that 4) holds, first observe that since 4) holds for $\mathcal T$, we only need to check the condition on the components $L$ and $L'$ affected by the exchange. Observe if there is a vertex $z \in \X \cap V(K'(\mathcal T))$ such that $zz' \in E(\redForest(\mathcal T))$, then $z = u(\mathcal T_a)$ since 4) holds for $\mathcal T$. 
    Since $u(\mathcal T_a)$ is in $L'$ and $y$ is in $L$, we have that 4) also holds for $\mathcal T'$.\\
    Note that either $K'(\mathcal T') = \varnothing$ or $K'(\mathcal T') = L$.
    As 5) holds for $\mathcal T$ and since $y$ is the only vertex of $\X \cap V(L)$ contained in $\pathInBlueTree{\mathcal T}{v}{y}$, we have that 5) also holds for $\mathcal T'$. Thus, $\mathring S(\mathcal T'_a) = (\mathring S(\mathcal T_a) \cap V(L)) - y$ if $K'(\mathcal T) \neq \varnothing$.\\
     Next, we want to show 6a) and suppose that $a = v_*$ is special for $\mathcal T$. Since $\pathInBlueTree{\mathcal T}{v}{y} \subseteq T_y(\bar S(\mathcal T_a), \mathcal T)$ and the vertex set of the supergraph is disjoint from $V(T_{v_*}(\bar S(\mathcal T_a), \mathcal T))$, we have that $T_{v_*}(\bar S(\mathcal T_a), \mathcal T) \subseteq T_{v_*}(\bar S(\mathcal T'_a), \mathcal T')$. Thus, $v_*$ is also special for $\mathcal T'$ and 6a) holds for $\mathcal T'_a$.\\
    Now we consider 6b). Suppose that $a = w(K)$ is not special for $\mathcal T$ and thus, there is no special vertex for $\mathcal T$ and $\IEqISigK$. Then $\pathInBlueTree{\mathcal T}{w(K)}{r}$ does not contain a vertex of $\X$ and in particular, it does not contain $x$. As all tails of the arcs of $E(\blueTree(\mathcal T')) \setminus E(\blueTree(\mathcal T))$ are descendants of $x$, we have that $\pathInBlueTree{\mathcal T}{w(K)}{r} \subseteq \blueTree(\mathcal T')$. Thus, there is no special vertex for $\mathcal T'$. Hence, 6b) holds for $\mathcal T'_a$.\\
    We show that $a \notin V(L)$ and thus, 7) holds for $\mathcal T'_a$: 
    if $u(\mathcal T_a)$ extends $K'(\mathcal T)$ in $\mathcal T$, then $a$ is not contained in $K'(\mathcal T)$ since 7) holds for $\mathcal T$.
    If, on the other hand, we have $u(\mathcal T_a) = a$, then $a \in V(L')$. Thus, 7) holds for $\mathcal T'_a$.\\
    Note that if $\varnothing \neq K'(\mathcal T') = L$, then $u(\mathcal T'_a) = y$ extends $L$ in $\mathcal T'$ and thus, 8) holds for $\mathcal T'_a$. Furthermore, note that $\pathInBlueTree{\mathcal T'}{y}{r}$ contains $(v, v')$. By the definition of $v$ and $v'$, we have that 9) holds for $\mathcal T'_a$.\\
    Thus, $\mathcal T'_a \in \setOfValidDecomps$ and furthermore, since $K'(\mathcal T') \in \{L, \varnothing\}$, we have that $K'(\mathcal T') \subseteq L \subseteq K'(\mathcal T) - vv'$, which completes the claim in this case.

    \doubleFigure{\subCaseOneTwoBefore}{\subCaseOneTwoAfter}{
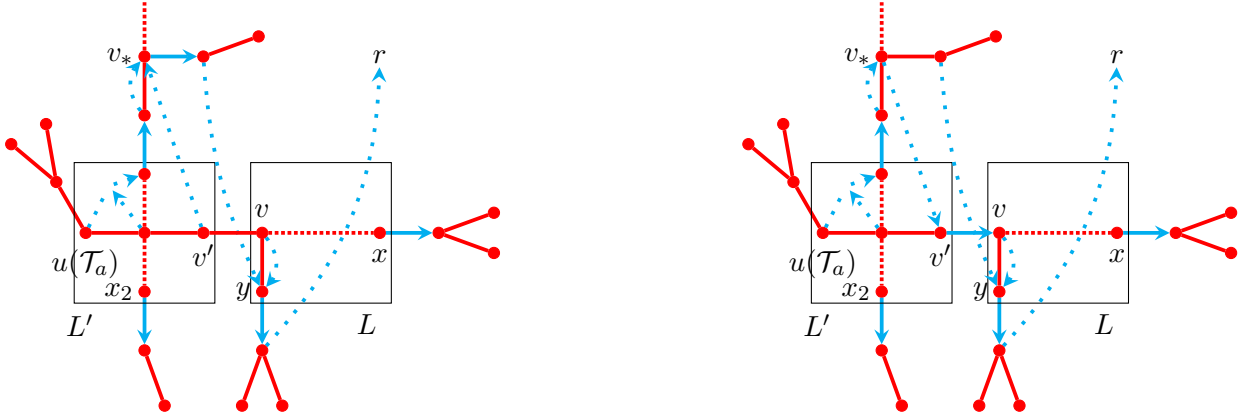
        $\mathcal T_a$ (left side) and $\mathcal T'_y$ (right side) in Subcase 1.2 in the proof of Lemma \ref{lemma:mainAugmentation}
    }{\label{fig:subcaseOneTwo}}
    
    \textbf{Subcase 1.2: $v'$ is a descendant of $y$ in $\blueTree(\mathcal T)$:}\\
    This case is depicted in Figure \ref{fig:subcaseOneTwo}. In this case, it follows that $a = v_*$ is special for $\mathcal T$ and $\pathInBlueTree{\mathcal T}{v'}{y}$ contains $v_*$. Note that $v$ is not a descendant of $v_*$ in $\blueTree(\mathcal T)$ or otherwise $\blueTree(\mathcal T)$ contains a cycle. Thus, we may obtain $\mathcal T'$ from $\mathcal T$ by performing $(v_*, v'_*) \leftrightarrow v'v$ (which reorients $\pathInBlueTree{\mathcal T}{v'}{v_*}$) and show that $(\mathcal T', y)$ is a valid intermediate state for $K$ and $\ell'$ and satisfies the lemma. 
    Note that 1), 2a), 2b) and 3) hold for $\mathcal T'$ by Lemma \ref{lemma:legalOrderInValidDecomp}.\\
    Now we prove 4) holds. As 7) holds for $\mathcal T_a$, we have that the component of $v_*$ in $\redForest(\mathcal T')$ contains exactly one edge of $\generatingEdgesInRed$, which is $v_* v'_*$, and thus, 4) holds for $\mathcal T'$.\\
    Note that if $K'(\mathcal T') \neq \varnothing$, then $K'(\mathcal T') = L$.
    As 5) holds for $\mathcal T$ and since $v_*$ is the only vertex of $\X \cap V(L)$ contained in $\pathInBlueTree{\mathcal T}{v'}{v_*}$, we have that 5) also holds for $\mathcal T'$.\\
    We have that $V(T_{v_*}(\bar S(\mathcal T_a), \mathcal T)) \subseteq V(T_y(\bar S(\mathcal T'_y), \mathcal T'))$ since in $\blueTree(\mathcal T')$ there is a path from $v_*$ to $y$ going over $(v', v)$ and this path does not contain any vertices of $\mathring S(\mathcal T'_y) - y$. Thus, $y$ is a special vertex for $\mathcal T'$ and 6a) holds for $\mathcal T'_y$. As there is a special vertex for $\mathcal T'$, 6b) trivially holds for $\mathcal T'_y$.\\
    Note that if $\overline{K'(\mathcal T)}$ contains an edge of $\generatingEdgesInRed$, then $u(\mathcal T_a)$, which is in $L'$, extends $\mathcal K'(\mathcal T)$ in $\mathcal T$. Since $y \in V(L)$, we have that 7) holds for $\mathcal T'_y$.\\
    Using the same argument, we have that there is no vertex extending $L$ in $\mathcal T'$, and thus, $u(\mathcal T'_y) = y \in V(L)$. Hence, 8) and 9) hold for $\mathcal T'_y$.\\
    Thus, $\mathcal T'_y \in \setOfValidDecomps$ and furthermore, since $K'(\mathcal T') \in \{L, \varnothing\}$, we have that $K'(\mathcal T') \subseteq L \subseteq K'(\mathcal T) - vv'$.
    
    \textbf{Case 2: $e(L') \geq \maxOfDAndEOfK - \ell'$:}\\
    Recall that $\pathIn{K'(\mathcal T)}{u(\mathcal T_a)}{x} = [v_1, \ldots, v_n]$.
    Let $j \in \{1, \ldots, n\}$ be maximal such that $v_j$ is not a descendant of $x$ in $\blueTree(\mathcal T)$ or $\pathInBlueTree{\mathcal T}{v_j}{r}$ visits $v_*$ and after that it visits $x$. Note that $j$ exists and $i-1 \leq j < n$ . Let $\bar v := v_{j+1}$ and $\bar v' := v_j$. Furthermore, let $M$ and $M'$ be the components in $K'(\mathcal T) - \bar v \bar v'$ of $x$ and $u(\mathcal T_a)$, respectively. Note that $e(M') \geq e(L')$. We note that the edge exchange operations in the following two subcases are similar to the ones in the Subcases 1.1 and 1.2.

    \textbf{Subcase 2.1: $\bar v'$ is not a descendant of $x$ in $\blueTree(\mathcal T)$:}\\
    We obtain $\mathcal T'$ from $\mathcal T$ by performing $(x, x') \leftrightarrow \bar v \bar v'$ and show that $(\mathcal T', a)$ is a valid intermediate state for $K$ and $\ell'$ and satisfies the lemma.\\
    Note that 1), 2a), 2b) and 3) hold for $\mathcal T'$ by Lemma \ref{lemma:legalOrderInValidDecomp}.\\
    Since $x$ is not in the component of $u(\mathcal T_a)$ in $\redForest(\mathcal T')$, we have that 4) holds for $\mathcal T'$.\\
    Note that $K'(\mathcal T') = M'$.
    As 5) holds for $\mathcal T$ and since by the definition of $x$ and $\bar v$ we have that $x$ is the only vertex of $\X \cap V(M)$ contained in $\pathInBlueTree{\mathcal T}{\bar v}{x}$, we have that 5) also hold for $\mathcal T'$.\\
    Analogously to Subcase 1.1 it can be proven that 6a) and 6b) hold for $\mathcal T'_a$.\\
    We show that $a \notin V(M)$ and thus, 7) holds for $\mathcal T'_a$: 
    if $u(\mathcal T_a)$ extends $M$ in $\mathcal T'$, then $a$ is not contained in $K'(\mathcal T)$ since 7) holds for $\mathcal T_a$. 
    If, on the other hand, we have that $u(\mathcal T_a) = a$, then $a \in V(M')$. Thus, 7) holds for $\mathcal T'_a$.\\
    Observe that 8) holds for $\mathcal T'_a$ since $u(\mathcal T'_a) = u(\mathcal T_a) \in V(M')$.\\
    Finally, suppose that 9) does not hold for $\mathcal T'_a$ and let $(z, z')$ be the first arc on $\pathInBlueTree{\mathcal T}{u(\mathcal T_a)}{r}$ that is not on $\pathInBlueTree{\mathcal T'}{u(\mathcal T_a)}{r}$. Then $(z, z')$ is in $\pathInBlueTree{\mathcal T}{\bar v}{x}$. 
    As 9) holds for $\mathcal T_a$, we have that $a = v_*$ is special for $\mathcal T$ and $v_* \in V(\pathInBlueTree{\mathcal T}{u(\mathcal T_a)}{x})$. Since $v_* \notin V(\pathInBlueTree{\mathcal T}{\bar v}{x})$, we have that $v_*$ is in $\pathInBlueTree{\mathcal T}{u(\mathcal T_a)}{z} \subseteq \blueTree(\mathcal T')$ and thus, 9) holds for $\mathcal T'_a$.\\
    Thus, $\mathcal T'_a \in \setOfValidDecomps$ and furthermore, $K'(\mathcal T') = M' \subseteq K'(\mathcal T) - \bar v \bar v'$.
    
    \textbf{Subcase 2.2: $\bar v'$ is a descendant of $x$ in $\blueTree(\mathcal T)$:}\\
    Then $a = v_*$ is special for $\mathcal T$ and $\pathInBlueTree{\mathcal T}{\bar v'}{x}$ contains $v_*$.
    Note that $\bar v$ is not a descendant of $v_*$ in $\blueTree(\mathcal T)$ or otherwise $\blueTree(\mathcal T)$ contains a cycle. Thus, we may obtain $\mathcal T'$ from $\mathcal T$ by performing $(v_*, v'_*) \leftrightarrow v'v$ and show that $(\mathcal T', x)$ is a valid intermediate state for $K$ and $\ell'$ and satisfies the lemma.\\
    We have that 1), 2a), 2b) and 3) hold for $\mathcal T'$ by Lemma \ref{lemma:legalOrderInValidDecomp}.\\
    As 7) holds for $\mathcal T_a$, we have that the component of $v_*$ in $\redForest(\mathcal T')$ contains exactly one edge of $\generatingEdgesInRed$, and thus, 4) holds for $\mathcal T'$.\\
    Note that $K'(\mathcal T') = M'$. Since $v_*$ is a descendant of $x$ in $\blueTree(\mathcal T)$, we have by the definition of $x$ that there are no vertices of $\mathring S(\mathcal T_a)$ on $\pathInBlueTree{\mathcal T}{\bar v'}{v_*}$, the path that is reoriented in $\blueTree(\mathcal T')$. Thus, since 5) holds for $\mathcal T$, we have that 5) also holds for $\mathcal T'$.\\
    We have that $V(T_{v_*}(\bar S(\mathcal T_a), \mathcal T)) \subseteq V(T_x(\bar S(\mathcal T'_x), \mathcal T'))$ since in $\blueTree(\mathcal T')$ there is a path from $v_*$ to $x$ going over $(\bar v', \bar v)$ and this path does not contain any vertices of $\mathring S(\mathcal T'_x) - x$. Thus, $x$ is a special vertex for $\mathcal T'$ and 6a) holds for $\mathcal T'_x$ and furthermore, 6b) trivially holds for $\mathcal T'_x$.\\
    Note that if $\overline{K'(\mathcal T)}$ contains an edge of $\generatingEdgesInRed$, then $u(\mathcal T_a)$, which is in $M'$, extends $K'(\mathcal T)$ in $\mathcal T$. Since $x \in V(M)$, we have that 7) holds for $\mathcal T'_x$.\\
    Using the same argument, we have that there is no vertex extending $M$ in $\mathcal T'$, and thus, $u(\mathcal T'_x) = x \in V(M)$. Hence, 8) and 9) hold for $\mathcal T'_x$.\\
    Thus, $(\mathcal T', x) \in \setOfValidDecomps$ and furthermore, $K'(\mathcal T') = M' \subseteq K'(\mathcal T) - \bar v \bar v'$.
\end{proof}

\begin{notation} \label{notation:alpha}
    In the following let $\alpha := \max \{d' - e(K), 0\}$.
\end{notation}

\begin{lemma} \label{lemma:existenceValidColouringForAugment}
    Let $\ell' \inZeroToEll$. If $|\mathcalC| \geq \ell' - \alpha + 2$, then there is a tuple $(\mathcal T, a) \in \setOfValidDecomps$ such that $K'(\mathcal T) = \varnothing$.
\end{lemma}
\begin{proof}
    Recall that by Observation \ref{obs:TStarValidDecomp} we have that $\setOfValidDecomps \neq \varnothing$.
    Suppose the lemma is not true and let $(\mathcal T, a) \in \setOfValidDecomps$ such that $e(K'(\mathcal T))$ is minimal.
    By Lemma \ref{lemma:mainAugmentation} we have that $\mathring S(\mathcal T_a) = \varnothing$.
    As 4) and 5) hold for $\mathcal T$ we have that $\X \cap V(K'(\mathcal T)) \subseteq \{a, t\}$ where $t$ is a vertex such that $tt' \in E(\redForest(\mathcal T))$. Since 7) holds for $\mathcal T_a$, we have that $|\X \cap V(K'(\mathcal T))| \leq 1$. Thus, $|\X \setminus V(K'(\mathcal T))| \geq \ell' - \alpha + 1$ and $e(K'(\mathcal T)) \leq e(K) - (\ell' - \alpha + 1) \leq \maxOfDAndEOfK - \ell'$, which is a contradiction to the definition of $K'(\mathcal T)$.
\end{proof}

Note that the lower bound of $|\mathcalC|$ cannot easily be lowered. In the ``worst case'' the size of $K'$ is only decreased by one edge in every exchange operation. An example for this can be seen in Figure \ref{fig:tightnessExample}.

\tripleFigure{\xyzPicOne}{\xyzPicTwo}{\xyzPicThree}{
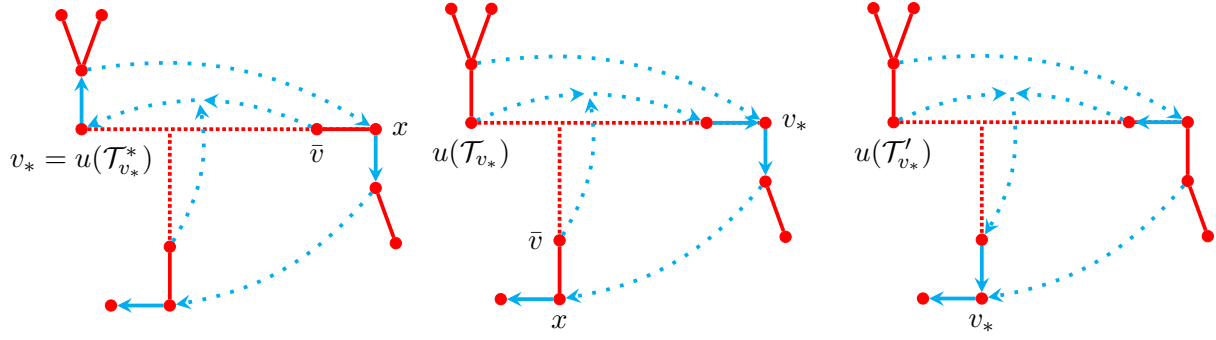
    An example where $\ell' = 2, \alpha = 0, |\mathcal C_2(K)| = 3$ showing the tightness of Lemma \ref{lemma:existenceValidColouringForAugment} when applying the augmentations of Lemma \ref{lemma:mainAugmentation} starting with $\mathcal T^*$ on the left and progressing to the right with every operation suggested by the proof of Lemma \ref{lemma:mainAugmentation}. Note that every vertex on the dotted red paths is supposed to have a path to $u(\mathcal T^*_{v_*})$. After the two augmentations we have that $K'(\mathcal T') \neq \varnothing$ has $e(K) - 2$ edges.
}{\label{fig:tightnessExample}}

Finally, we show that the decomposition $\mathcal T$ from Lemma \ref{lemma:existenceValidColouringForAugment} either already contradicts the minimality of $\sigma^*$ or there is a special path for $\mathcal T$ that we can use to retrieve a decomposition having a smaller legal order than $\sigma^*$.

\begin{lemma} \label{lemma:numSmallChildren}
    For every $b \inOneToK$ and $\ell' \inZeroToEll$ the number of children of $K$ generated by $\blueTree(\mathcal T^*)$ each having at most $\ell'$ edges is at most $\ell' - \alpha + 1$.
\end{lemma}
\begin{proof}
    Let $\ell' \inZeroToEll$ and suppose that $K$ has $\ell' - \alpha + 2$ children, each having at most $\ell'$ edges. 
    Thus, by Lemma \ref{lemma:sumOfChildRelation} we have $e(K) \geq d' - \ell'$ and $\alpha \leq \ell'$.
    By Lemma \ref{lemma:existenceValidColouringForAugment} there is a tuple $(\mathcal T, a) \in \setOfValidDecomps$ such that for every $L \in \mathcal L(\mathcal T)$ we have $e(L) \leq \edgeLimit < e(K)$. If $a$ is special for $\mathcal T$, let $v_* := a$.
    Let $R' := \redForest(\mathcal T) + (v_*, v'_*)$ if $a$ is special, and $R' := \redForest(\mathcal T)$ otherwise.
    
    \begin{claimUnnumbered}
        $\rho(R') \leq \rho^*$
    \end{claimUnnumbered}
    \begin{proofInProof}
        Let $A$ be the set of components of $R'$ that contain a vertex of $K$. 
        Note that every component $\bar L \in A$ has $e(\bar L) \leq \maxOfDAndEOfK$ including the component of $v_*$ (if $a$ is special for $\mathcal T$) since 7) holds for $\mathcal T_a$. If $e(\bar L) \leq \max \{d', e(K) - 1\}$ for every $\bar L \in A$, we have that $\rho(R') \leq \rho^*$. Thus, suppose that there is an $\bar L \in A$ such that $e(\bar L) = e(K) > d'$. Then $|E(\bar L) \cap E(K)| \geq e(K) - \ell' - 1$ and thus, $|(E(K) \cap E(R')) \setminus E(\bar L)| \leq \ell'$ since at least one edge of $E(K)$ is coloured blue in $\mathcal T$. Thus, for every $\bar L' \in A \setminus \{\bar L\}$ we have that $e(\bar L') \leq \ell' + (\ell' + 1) \leq d'$ by Observation \ref{obs:twoEllPlusOne} and the claim follows.
    \end{proofInProof}
    
    Suppose that $a = w(K)$ is not special for $\mathcal T$. Note that since 7) holds for $\mathcal T_a$, the component of $w(K)$ in $\redForest(\mathcal T)$ is contained in $\mathcal L(\mathcal T)$ and thus has less than $e(K)$ edges. Since 2b) holds for $\mathcal T$, there is a smaller legal order than $\sigma^*$ for $\mathcal T$, which is a contradiction.\\
    Thus, $a$ is special for $\mathcal T$. Let $\sigma$ be the legal order for $\mathcal T$ mentioned in 2a). Note that there is a minimal special path $P := [v_0, \ldots, v_l]$ with respect to $\mathcal T, \sigma$ and $(v_*, v'_*)$ such that $\iSigStar(v_0) \leq \iSigStar(T_{v_*}(\bar S(\mathcal T_a), \mathcal T)) = \I$. If $\IEqISigK$, then $\pathInBlueTree{\mathcal T}{w(K)}{v_*}$ preceded by the witnessing arc is a special path and thus, we even have $\iSigStar(v_0) \leq \I - 1$ in this case.
    But this contradicts Corollary \ref{cor:specialPathInStandardCase}.
\end{proof}

\section{Density Calculations}
\label{sec:finish}
In this section we show that the density of a component $K$ of $\explSG$ and its small children is on average at least $\density$, which immediately enables us to prove Theorem \ref{thm:approxSNDTC}.
Before this we need a few technical lemmas which mostly amount to rearranging equations.

\begin{lemma} \label{lemma:minusAlphaTermNotDense}
    Let $1 \leq n \leq \ell$. 
    Then
    \[
    \frac{n + k\frac{(n-1)n}{2}}{n + k\frac{(n-1)n}{2} + kn}
    < \frac{d}{d+k+1}.
    \]
\end{lemma}
\begin{proof}
We have:
    \begin{align*}
        \frac{n + k\frac{(n-1)n}{2}}{n + k\frac{(n-1)n}{2} + kn}
        = \frac{\frac{1}{k} + \frac{n-1}{2}}{\frac{1}{k} + \frac{n-1}{2} + 1}
        \leq \frac{\frac{n+n}{2}}{\frac{n+n}{2} + 1}
        \leq \frac{\ell}{\ell+1}
        < \frac{d}{d+k+1}.
    \end{align*}

    Here, in the first equality we cancelled out $n$ and $k$. In the second inequality we set $k = 1$, and use the fact that $n \geq 1$. The third inequality follows since $n \leq \ell$. The fourth inequality follows by similar reasoning to Observation \ref{obs:ellSmallerDensity}.
\end{proof}

\begin{obs} \label{obs:compareDPrimeWithBetterBound}
    We have that
    \[
    d'
    \geq d \cdot \frac{k(\ell + 1) + 1}{k+1} - \frac{k}{2} \ell (\ell + 1)
    .
    \]
\end{obs}
\begin{proof}

    \begin{align*}
    d'
    &= d + \biggCeil{k \ell \BigParan{\frac{d}{k+1} - \frac{1}{2}\bigParan{\ell + 1}}}\\
    &= d + \biggCeil{\frac{k \ell d}{k+1}} - \frac{k}{2}  \ell \bigParan{\ell + 1}\\
    &\geq d \cdot \frac{k(\ell + 1) + 1}{k+1} - \frac{k}{2} \ell (\ell + 1)
    .
    \end{align*}
\end{proof}
\begin{obs} \label{obs:densityMinus}
    Let $a_1, a_2, b_1, b_2 \in \mathbb N$ where $a_1 \geq a_2$, $b_1 > b_2 > 0$ and $\frac{a_1}{b_1} \geq \frac{a_2}{b_2}$. Then $\frac{a_1-a_2}{b_1-b_2} \geq \frac{a_1}{b_1}$.
\end{obs}

\begin{lemma} \label{lemma:densityKPlusChildren}
    For every red component $K \neq R^*$ of $\explSG$ that is not small, we have that
    \[
    \frac{e(K) + \sum_{C \in \mathcal C(K)} e(K)}{v(K) + \sum_{C \in \mathcal C(K)} v(K)}
    \geq \density.
    \]
\end{lemma}
\begin{proof}
    First, suppose that $K$ does not have small children. Then $e(K) \geq \ell+1$ by Lemma \ref{lemma:sumOfChildRelation} and the lemma holds by Observation \ref{obs:ellSmallerDensity}.\\
    Thus, let $K$ have a small child and suppose that the lemma is not true. In the following $\alpha$ is defined as in Notation \ref{notation:alpha}. We enumerate some facts that we will use.
    \begin{enumerate}[a)]
        \item For any $\ell' \inZeroToEll$, $K$ has at most $k (\ell' - \alpha + 1)$ small children each having at most $\ell'$ edges by Lemma \ref{lemma:numSmallChildren}.
        \item By Lemma \ref{lemma:sumOfChildRelation} we have that every small child of $K$ has at least $\alpha$ edges.
        \item Let $\ell' \inZeroToEll$. Note that adding $\ell'$ to the nominator and $\ell' + 1$ to the denominator of a fraction that is smaller than $d/(d+k+1)$ again results in a smaller fraction than $d/(d+k+1)$ since $\ell / (\ell + 1) < d/(d+k+1)$.
        \item Note that for $1 < a < b$ we have that $\frac{a-1}{b-1} < \frac{a}{b}$. Thus, informally speaking, decreasing the size of a small child also decreases
        $\frac{e(K) + \sum_{C \in \mathcal C(K)} e(K)}{v(K) + \sum_{C \in \mathcal C(K)} v(K)}$.
    \end{enumerate}
    Now, we prove the lemma. In the first ``$\geq$'' we use b), the (possibly repeated) usage of c) and d) while still being compliant with a).
    
    \begin{align*}
        \frac{d}{d+k+1} 
        &> \frac{e(K) + \sum_{C \in \mathcal C(K)} e(K)}{v(K) + \sum_{C \in \mathcal C(K)} v(K)}\\
        &\geq \frac{e(K) + k\sum_{i = \alpha}^\ell i}
        {e(K) + 1 + k\sum_{i = \alpha}^\ell (i+1)}\\
        &\geq \frac{d' - \alpha     + k\sum_{i = \alpha}^\ell i}
        {d' - \alpha + 1 + k\sum_{i = \alpha}^\ell i + k(\ell + 1) - k\alpha}\\
        \overset{\text{Obs.\ \ref{obs:compareDPrimeWithBetterBound}}}&{\geq}
        \frac{
            d \cdot \frac{k(\ell + 1) + 1}{k+1}                - (\alpha + k \sum_{i=0}^{\alpha-1} i)
        }
        {
            d \cdot \frac{k(\ell + 1) + 1}{k+1}+(k(\ell+1)+1)  - (\alpha + k \sum_{i=0}^{\alpha-1} i + k\alpha)
        }\\
        \overset{\begin{subarray}{c}\text{Lemma \ref{lemma:minusAlphaTermNotDense},}\\ \text{Obs.\ \ref{obs:densityMinus}}\end{subarray}}&{\geq} 
        \frac{d}{d+k+1},
    \end{align*}
    which is a contradiction.
\end{proof}

Now, we are able to prove Theorem \ref{thm:approxSNDTC}.

\begin{proofOfMainTheorem}
    Let $\mathcal K$ be the set of red components of $\explSG$ that are not small. By Lemma \ref{lemma:smallNoSmallChildren} we have that the set of red components of $\explSG$ is
    \[
    \bigcup_{K \in \mathcal K} (\{K\} \cup \bigcup_{C \in \mathcal C(K)} \{C\}).
    \]
    Note that in this representation every component is enumerated exactly once.
    Thus, by Lemma \ref{lemma:rootNoChildren} we have that
    \[
    \frac{e(R^*) + \sum_{C \in \mathcal C(R^*)} e(C)}{v(R^*) + \sum_{C \in \mathcal C(R^*)} v(C)}
    = \frac{e(R^*)}{e(R^*) + 1}
    > \frac{d'}{d' + 1}
    > \density.
    \]
    By Lemma \ref{lemma:densityKPlusChildren} we have that
    \[
    \frac{e_r(\explSG)}{v(\explSG)}
    = \frac{\sum_{K \in \mathcal K} \big(e(K) + \sum_{C \in \mathcal C(K)}e(C)\big)}{\sum_{K \in \mathcal K} \big(v(K) + \sum_{C \in \mathcal C(K)}v(C)\big)}
    > \density,
    \]
    which is a contradiction to Lemma \ref{lemma:densityRedExplSG}.
\end{proofOfMainTheorem}

\printbibliography

\end{document}